\documentclass[12pt]{amsart}
\usepackage[babel]{csquotes}
\usepackage{enumitem}
\usepackage{amsmath,amsthm,amssymb,mathrsfs,amsfonts,verbatim,enumitem,color,leftidx}
\usepackage{mathabx}
\usepackage[left=2.9cm,right=2.9cm,top=3.3cm,bottom=3.4cm,a4paper]{geometry}
\usepackage{mathtools}
\usepackage{lipsum}
\usepackage{etoolbox} 
\usepackage{bbm}
\usepackage[all,tips]{xy}
\usepackage{here}
\usepackage{graphicx,ifpdf}
\usepackage{stmaryrd}
\ifpdf
\fi

\usepackage[right,displaymath,mathlines]{lineno}

\definecolor{darkolivegreen}{rgb}{0.33, 0.42, 0.18}
\definecolor{cadmiumgreen}{rgb}{0.0, 0.42, 0.24}
\definecolor{calpolypomonagreen}{rgb}{0.12, 0.3, 0.17}
\usepackage[colorlinks]{hyperref}
\hypersetup{
linkcolor=blue,          
citecolor=cadmiumgreen,      
}

\usepackage{tikz}
\usetikzlibrary{arrows,snakes,backgrounds}
\usetikzlibrary{arrows.meta,calc,patterns}

\newtheorem{thm}{Theorem}[section]
\newtheorem{lem}[thm]{Lemma}
\newtheorem{coro}[thm]{Corollary}
\newtheorem{prop}[thm]{Proposition}

\theoremstyle{definition}

\theoremstyle{remark}

\numberwithin{equation}{section}

\definecolor{esperance}{rgb}{0.0,0.5,0.0}

\usetikzlibrary{arrows,snakes,backgrounds}

\newcommand{\onto}{\xymatrix{\ar@{>>}[r]&}}
\begin{document}

\title[Edge zeta function of $\operatorname{PGL}_3$]{Edge zeta function and closed  cycles in the standard non-uniform complex from $\operatorname{PGL}_3$}
\author{Soonki Hong}
\address{Department of Mathematics\\Pohang University of Science and Technology\\Cheongam-ro, Pohang 37673, Republic of Korea}
\email{soonkihong@postech.ac.kr}
\author{Sanghoon Kwon}
\address{Department of Mathematical Education\\Catholic Kwandong University\\Beomil-ro, Gangneung 25601, Republic of Korea}
\email{skwon@cku.ac.kr, shkwon1988@gmail.com}

\maketitle

\begin{abstract}
In this paper, we define the edge zeta function of weighted complex. We also present the formula for the edge zeta function of the standard non-uniform complex $\operatorname{PGL}(3,\mathbb{F}_q[t])\backslash\operatorname{PGL}(3,\mathbb{F}_q(\!(t^{-1})\!))/\operatorname{PGL}(3,\mathbb{F}_q[\![t^{-1}]\!])$, arising from the group $\operatorname{PGL}_3$, as a rational function. Applying trunction in a specific direction is one of the main ingredient. As a result, we obtain the exact formula for the number of closed cycles coming from geodesics in the building.
 \end{abstract}

\tableofcontents

\section{Introduction}\label{sec:1}

The concept of a zeta function was first introduced by Bernhard Riemann in 1859, primarily in the context of number theory, with the Riemann zeta function being crucial for understanding the distribution of prime numbers. This progress in number theory inspired the exploration of analogous zeta functions in other mathematical areas. One of the notable developments came from Selberg, who extended the idea into the context of hyperbolic geometry. His trace formula relates the spectrum of the Laplace operator to closed geodesics.

In the 1960s, Ihara \cite{I} developed what is now known as the Ihara zeta function for discrete groups acting on trees, inspired by parallels between number theory and the theory of finite graphs. The Ihara zeta function for split rank one $p$-adic groups is an analogue of the Dedekind zeta function of an algebraic number field, and equals to the Hasse-Weil zeta function of the corresponding Shimura curve. Later, Serre observed that Ihara's definition could be extended to arbitrary finite graphs.
 
Revisiting Ihara's idea, H. Bass provided a determinant formula for the Ihara zeta function in 1980s. This determinant formula, known as the Bass-Ihara formula, established a connection between algebraic graph theory and spectral graph theory. The formula simplifies the computation of the zeta function and reveals important connections to the eigenvalues of the adjacency matrix. After Bass's work, the Ihara zeta function became a tool for studying regular and irregular finite graphs, leading to results on graph isomorphism, spanning trees, and graph spectra.

%Ihara provided in \cite{I} the link between geometric and arithmetic zeta functions by showing that the Ihara zeta function for a finite arithmetic quotient of a Bruhat-Tits tree equals the Hasse-Weil zeta function of the corresponding Shimura curve.

The development of zeta functions of finite graphs has indeed been generalized in two primary directions. One direction is to consider fields with positive characteristic $F$ rather than $p$-adic fields, and to study non-cocompact groups $\Gamma$ of a rank one algebraic group $G$ over $F$. It is related to extending finite graphs to the case of infinite weighted graphs. For example, \cite{CS} and \cite{Cl} made use of a finite trace argument on a group von Neumann algebra and define the zeta function as a determinant. Another approach, described in \cite{GZ}, involves approximating an infinite graph by finite ones and defining the zeta function as a suitable limit. In \cite{CJK}, the Ihara zeta function is extended to infinite graphs by considering cycles that pass through a specific point using Heat kernel, rather than considering all cycles. Furthermore, in \cite{LPS}, the zeta function is extended to the case of an infinite graph equipped with an invariant measure and a groupoid action, which is referred to as a measure graph. Meanwhile, \cite{DK18} take another approach by considering cycles obtained from geodesics in the universal covering tree of infinite weigted graphs.

Another direction focuses on higher-dimensional complexes and discrete subgroups of higher rank groups. It is discussed in \cite{DK14} the geometric zeta functions for higher rank $p$-adic group. %\cite{DH}, 
The authors in \cite{KL} obtained a closed form expression of the zeta function of a finite quotient of the building for $\operatorname{PGL}_3$ and presented two different expressions of the zeta function. They also determined some necessary and sufficient conditions for the complex to satisfy the Ramanujan condition. Similar results were also discussed in \cite{KLW10}, \cite{KLW18}, and \cite{DKM} through different perspectives. It should be noted that these papers deal with finite simplicial complexes.

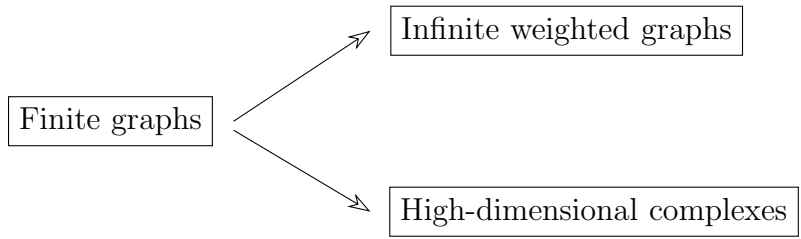
\begin{figure}[H]
\begin{center}
\begin{tikzpicture}[scale=0.6]
\draw[-{Stealth[open,length=3mm, width=2mm]}] (0,0)--(3,2);
\draw [-{Stealth[open,length=3mm, width=2mm]}] (0,-0.2) -- (3,-2);
\node [draw] at (-2.7,0) {Finite graphs};
\node [draw] at (7.3,2) {Infinite weighted graphs};
\node [draw] at (7.9,-2) {High-dimensional complexes};
\end{tikzpicture}
\end{center}
%\caption{Local transition probability of type 1 and type 2 edges}\label{figure4}
\end{figure}

In this paper, we investigate a non-compact higher rank case. One problem in higher rank is the lack of a unified zeta function, as in higher dimensional buildings there are several possibilities to generalize Ihara's approach. We take geometric approach and investigate the edge zeta function of a non-cocompact arithmetic lattice of $\operatorname{PGL}_3$.

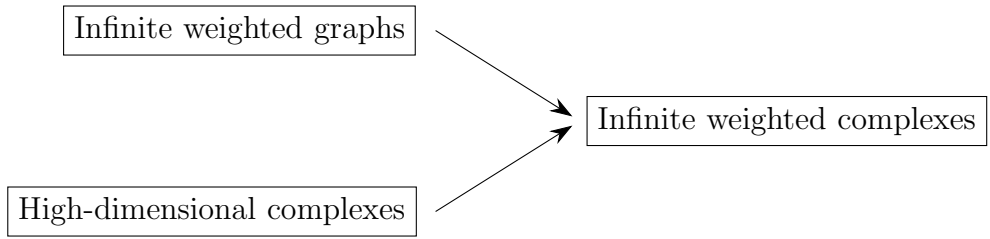
\begin{figure}[H]
\begin{center}
\begin{tikzpicture}[scale=0.6]
\draw[-{Stealth[length=3mm, width=2mm]}] (6,2)--(9,0.1);
\draw [-{Stealth[length=3mm, width=2mm]}] (6,-2) -- (9,-0.1);
\node [draw] at (1.7,2) {Infinite weighted graphs};
\node [draw] at (1.1,-2) {High-dimensional complexes};
\node [draw] at (13.7,0) {Infinite weighted complexes};
\end{tikzpicture}
\end{center}
%\caption{Local transition probability of type 1 and type 2 edges}\label{figure4}
\end{figure}

Before explaining the geometric zeta function we are using, we would like to introduce some terms related to graphs and complexes used in this paper, and the geometric ideas on weighted graphs to help guide the readers. In this paper, all edges used in a graph or simplicial complex refer to \emph{directed edges}. We will denote the starting vertex (source) and terminal vertex (target) of an edge $e$ as $s(e)$  and $t(e)$, respectively.

The idea of \cite{DK18} is considering cycles that come from geodesics in the universal covering. Given a tree lattice $\Gamma$ of a uniform tree $\mathcal{T}$, the Bass-Ihara zeta function for $\Gamma$ is defined as follows. Let $X=\Gamma\backslash \mathcal{T}$ be the quotient graph. A path $p=(e_1,\ldots,e_n)$ in $X$ is called \emph{closed} if $t(e_n)=s(e_1)$. Then the shifted path $\sigma(p)=(e_2,\ldots,e_n,e_1)$ is closed again, and the shift induces an equivalence relation on the set of closed paths, where two closed paths are equivalent if one is obtained from the other by finitely many shifts. A \emph{cycle} is an equivalence class of closed paths. If $c$ is a cycle, then any power $c^n$, $n\in\mathbb{N}$, obtained by running the same path for $n$ times is again a cycle, and a given cycle $c_0$ is called a \emph{prime cycle} if it is not a power of a shorter one. For every given cycle, there exists a unique prime cycle $c_0$ such that $c=c_0^m$ for some $m\in \mathbb{N}$. The number $m=m(c)$ is called the \emph{multiplicity} of $c$. A path $p=(e_1,\ldots,e_n)$ is called \emph{reduced} if $e_{j+1}\ne e_j^{-1}$ for every $1\le j\le n-1$. A cycle is called reduced if it consists only of reduced paths.

For an edge $\widetilde{e}$ of $\mathcal{T}$ and $e'$ of $X=\Gamma\backslash\mathcal{T}$, let
$$w(\widetilde{e},e')=\#\{\widetilde{e}'\in \operatorname{OE}(\mathcal{T})\colon \pi(\widetilde{e}')=e',(\widetilde{e},\widetilde{e}') \textrm{ is a reduced path in }\mathcal{T}\}.$$ It counts how many edges in $\mathcal{T}$ are connected to those in the quotient graph without backtracking. For an edge $e$ of $X=\Gamma\backslash \mathcal{T}$, we write $w(e,e')=w(\widetilde{e},e')$, where $\widetilde{e}$ is any preimage of $e$ in $\mathcal{T}$. This is well-defined since $w(e,e')$ does not depend on the specific choice of the preimage $\widetilde{e}$.

Now for a closed path $p=(e_1,e_2,\ldots,e_n)$, let 
$$w(p)=\prod_{j\textrm{ mod }n}w(e_j,e_{j+1}).$$ 
Then, the Bass-Ihara zeta function for $\Gamma$ is defined by
$$Z_\Gamma(q^{-s})=\prod_c\left(1-w(c)q^{-s\ell(c)}\right)^{-1},$$
where the product runs over all prime cycles, and $\ell(c)$ is the length of the cycle $c$. It is shown in \cite{DK18} that this product converges to a rational function for small $|q^{-s}|$ when $\Gamma$ is \emph{cuspidal} (or equivalently, \emph{geometrically finite}). For the definition of cuspidality, the readers may refer to \cite{DK18} or \cite{Pa}. Especially, the similar property have been already proved in \cite{B} when $X$ is compact.

%Let $F$ be a non-archimedean local field with $q$ elements in its residue field and $\mathcal{O}_F$ be the ring of integers of $F$. If we denote by $\pi$ a uniformizer of $\mathcal{O}_F$, then the residue field $\mathcal{O}_F/\pi\mathcal{O}_F$ is isomorphic to $\mathbb{F}_q$. Let $\Gamma$ be a discrete torsion-free cocompact subgroup of $\operatorname{PGL}_3(F)$ such that $\operatorname{ord}_\pi\det(\Gamma)\equiv 0$ mod $3$. The quotient $$X_\Gamma=\Gamma\backslash \mathcal{B}_3(F)=\Gamma\backslash\operatorname{PGL}_3(F)/\operatorname{PGL}_3(\mathcal{O}_F)$$ is a finite $(q+1)$-regular $2$-dimensional simplicial complex. The zeta function $Z(X_\Gamma,u)$ of such complexes $X_\Gamma$ that satisfies various desirable properties has been defined in previous works. For example, it is a rational function with a closed form expression, provides topological and spectral information of $X_\Gamma$, and satisfies RH if and only if $X_\Gamma$ is Ramanujan.

By extending this idea to higher dimensions, we may define the geometric edge zeta function for non-compact weighted complexes. Let $\mathcal{B}$ be a two dimensional simplicial complex and $X=\Gamma\backslash \mathcal{B}$ be a quotient complex of $\mathcal{B}$.
We need to consider a \emph{geodesic cycle of fixed type}. A geodesic cycle $c$ of $X$ is a closed path that lifts to a straight line in the universal covering building $\mathcal{B}$, which can be shown to lie in an apartment $\mathcal{A}$ of $\mathcal{B}$. Thus, it is a straight line in $\mathcal{A}$ starting at a vertex $v$ and ending at a vertex $\gamma v$ for some element $\gamma\in\Gamma$. If it is composed only of the directed edges of the same type $k\in\{1,2\}$, then we say $c$ is a geodesic cycle of type $k$. In a similar manner to the case previously mentioned, a geodesic cycle $c$ is primitive if it is not obtained by repeating a shorter geodesic cycle more than once. Two geodesic cycles are equivalent if they only differ by starting points. Denote by $[c]$ the equivalence class of $c$.
More detailed explanations are provided in Section~\ref{sec:2}. %\ref{sec:3} and \ref{sec:4}. 

As before, the primes of $X=\Gamma\backslash\mathcal{B}$ are equivalence classes of primitive geodesic cycles and the geometric length $\ell(c)$ of $c$ is the number of edges in $c$. For $k=1,2$, the type $k$ edge zeta function of $X$ is defined as
$$Z_{X,k}(q^{-s})=\prod_{[c]\colon\textrm{prime of type }k}\frac{1}{1-q^{-sk\ell(c)}}.$$%\exp\left(\sum_{n=1}^{\infty}\frac{N_n(X_\Gamma)}{n}q^{-sn}\right)=
and the edge zeta function $Z_X(q^{-s})$ for $X$ is defined by their product
$$Z_X(q^{-s})=Z_{X,1}(q^{-s})Z_{X,2}(q^{-s})=\prod_{k=1}^2\prod_{[c]\colon\textrm{prime of type }k}\frac{1}{1-q^{-sk\ell(c)}}.$$

%A geodesic cycle $C$ of $X_\Gamma$ is a closed $1$-dimensional geodesic path in $X_\Gamma$. It has a starting vertex and orientation. It lifts to a geodesic path $C'$ starting at a vertex in $\mathcal{B}_3(F)$, which can be shown to lie in an apartment $\mathcal{A}$ of $\mathcal{B}_3(F)$. Thus, it is a straight line in $\mathcal{A}$ starting at a vertex $v$ and ending at a vertex $\gamma v$ for some element $\gamma\in\Gamma$, using edges of the same type $i\in\{1,2\}$. In addition, a geodesic cycle $C$ is required to remain a geodesic cycle after changing the starting vertex. This is equivalent to $v,\gamma v$ and $\gamma^2v$ being collinear in $\mathcal{A}$. Such a $C$ is called a geodesic cycle of $X_\Gamma$ of type $i$. 

%For $n\ge 1$, let $N_n(X_\Gamma)$ denote the number of geodesic cycles in $X_\Gamma$ of type $1$ and geometric length $n$; it is also equal to that of type $2$. 

Similarly to the  case of \emph{weighted} graphs, we define the geometric zeta function of the discrete subgroup $\Gamma$ of $\operatorname{PGL}_3$ by
$$Z_\Gamma(q^{-s})=Z_{\Gamma,1}(q^{-s})Z_{\Gamma,2}(q^{-s})=\prod_{k=1}^2\prod_{c\colon\textrm{prime of type }k}\frac{1}{1-w(c)q^{-sk\ell(c)}}$$ where $\ell(c)$ is the length of the cycle $c$ and $w(c)$ is the same weight previously defined in the case of graphs. For a detailed definition, we refer to Section~\ref{sec:4}.

\begin{thm} For $\Gamma=\operatorname{PGL}(3,\mathbb{F}_q[t])$, the zeta function $Z_\Gamma(q^{-s})$ converges for $s$ with $\operatorname{Re}s>2$, and it is given by 
$$Z_{\Gamma}(q^{-s})=\frac{(1-q^{4-3s})^2(1-q^{4-6s})^2}{(1-q^{3-3s})(1-q^{6-3s})(1-q^{3-6s})(1-q^{6-6s})}.$$
\end{thm}
\begin{proof}[Idea of Proof] Given a cycle $c$, we denote by $c_0$ the underlying prime cycle of $c$. We may compute formally
\begin{align*}
Z_{\Gamma,k}(q^{-s})^{-1}=\exp\left(-\sum_{n=1}^{\infty}\frac{q^{-skn}}{n}\sum_{c\colon \ell(c)=n}w(c)\ell(c_0)\right).
\end{align*}
Let $\operatorname{E}_k(\Gamma\backslash \mathcal{B})$ be the set of all oriented edges in $\Gamma\backslash \mathcal{B}$ of type $k$ and define the operator $T_k\colon S(\operatorname{E}_k)\to S(\operatorname{E}_k)$ by
$$T_ke=\sum_{e'}w(e,e')e'$$ where the sum runs over all same type oriented edges $e'$ with $s(e')=t(e)$. %and the lifts of the vertices$s(e),s(e'),t(e')$ do not form a chamber in $\mathcal{B}$. 
We will show that for any given $n\in\mathbb{N}$, the operator $T_k^n$ is traceable, and the trace is given by
$$\operatorname{Tr}(T_k^n)=\sum_{c\colon\ell(c)=n}w(c)\ell(c_0).$$
Hence, we deduce that $Z_{\Gamma,k}(q^{-s})^{-1}=\exp\bigr(\!\operatorname{Tr}(\log(1-q^{-sk}T_k))\bigr)$. Using a Fubini trick and truncation by certain horizontal direction to actually compute this determinant, we obtain
$$Z_{\Gamma,k}(q^{-s})=\frac{(1-q^{4-3ks})^2}{(1-q^{3-3ks})(1-q^{6-3ks})}$$
for both $k=1,2$. This proves the result.
\end{proof}

Let us denote by $$N_m(\Gamma\backslash\mathcal{B})=\sum_{c\colon\ell(c)=m}w(c)\ell(c_0)$$ the number (with weights) of type $1$ closed cycles in $\Gamma\backslash \mathcal{B}$ of length $m$. In terms of the operator $T_k$, we have $N_m(\Gamma\backslash\mathcal{B})=\operatorname{Tr} T_k^m$ for all postive integers $m$. Hence, we have
$$Z_{\Gamma,k}(q^{-s})=\frac{1}{\det(I-q^{-sk}T_k)}=\exp\left(\sum_{m=1}^{\infty}\frac{N_m(\Gamma\backslash\mathcal{B})}{m}q^{-skm}\right)$$ which yields the following corollary. (See Section~\ref{sec:5} for the detail.)
\begin{coro}[Counting closed cycles]\label{coro:cycle_intro} Let $\Gamma$ be the discerete subgroup $\operatorname{PGL}(3,\mathbb{F}_q[t])$ of $\operatorname{PGL}(3,\mathbb{F}_q(\!(t^{-1})\!))$ and $N_m(\Gamma\backslash\mathcal{B})$ be defined as above. Then, we have
$$N_{m}=\left\{\begin{array}{ll}3q^{6r}-6q^{4r}+3q^{3r} & \textrm{ if }m=3r \\ 0 & \textrm{ otherwise}\end{array}\right..$$
\end{coro}

%--- 후반 어딘가에 

Although we did not define the chamber zeta function here, it can be defined for weighted chambers in a similar manner to that in \cite{KL}, \cite{KLW18}. They established the zeta identity relating the edge zeta function, the chamber zeta function, and the topological properties of finite complexes. It seems likely that this work could be extended to infinite complexes, and we expect that the results concerning this extension will be addressed future.

Generally, the Bruhat-Tits building of a non-archimedean local field can be viewed as an analogue of the symmetric space of a semi-simple Lie group. In the latter case, a Lefschetz formula has been developed, which expresses geometrical data of the geodesic flow and its monodromy in terms of Lie algebra cohomology, or more generally, foliation cohomology. This framework has been adapted to the context of $p$-adic groups in \cite{De} and applied in \cite{DK14}. It would be also interesting to demonstrate a similar result in positive characteristic setting.
%---

This paper is organized as follows. We introduce the Bruhat-Tits building for $\operatorname{PGL}_3$ and the geodesics within the building in Section~\ref{sec:2}. Section~\ref{sec:3} deals with the standard arithmetic quotient of $\operatorname{PGL}_3$ and introduces the notion of admissible paths, which are the key objects that will be used in the definition of the edge zeta function. In Section~\ref{sec:4}, the edge zeta function for weighted complexes is formally defined and a determinant formula for the zeta function will be derived, ensuring the convergence of the zeta function. We present the closed formula for the edge zeta function and the number of closed cycles in Section~\ref{sec:5}.

%%---
%% Section 2
%%---

\section{Building for $\operatorname{PGL}_3$, apartments and geodesics}\label{sec:2}

In this section, we describe the Bruhat-Tits building for $\operatorname{PGL}_3$ and the role of apartments and geodesics in this building. We will discuss how geodesics within the building correspond to certain paths in the quotient space later.

\subsection{Building $\mathcal{B}$ for $\operatorname{PGL}_3$} Let $F$ be a non-Archimedean local field with a discrete valuation $\nu$ and $\mathcal{O}$ be the valuation ring of $F$. Let $\pi$ be the uniformizer of $\mathcal{O}$ for which $\pi\mathcal{O}$ is the unique maximal ideal of $\mathcal{O}$. Let $Z=\{\lambda I: \lambda\in F\}$. We consider the $3\times 3$ projective general linear group $$\operatorname{PGL}(3,F)=GL(3,F)/Z$$ and let $PGL(3,\mathcal{O})$ be the image of the map from $GL(3,\mathcal{O})$ to $\operatorname{PGL}(3,F)$ defined by 
$$g\rightarrow gZ.$$ 

The \textit{Bruhat-Tits building} $\mathcal{B}$ associated to the group $\operatorname{PGL}(3,F)$ is the $2$-dimensinal contractible simplicial complex defined as follows. We say two $\mathcal{O}$-lattices $L$ and $L'$ of rank $3$ are in the same equivalence class if $L=sL'$ for some $s\in F^\times$. The set of the vertices of $\mathcal{B}$ will be the set of the equivalence classes $[L]$. For given $3$-vertices $[L_1],[L_2],[L_3]$, they form a $2$-dimensional simplex in $\mathcal{B}$ if 
\begin{equation}\label{eq:1.1}
\pi L_1'\subset L_3'\subset L_{2}'\subset L_1'
\end{equation}
for some $L_i'\in [L_i].$ 

Let $\mathcal{O}^3=\mathcal{O}\oplus\mathcal{O}\oplus\mathcal{O}$ be the standard $\mathcal{O}$-lattice of rank $3$. Since every scalar matrix $\lambda I$ preserves every equivalence class of $\mathcal{O}$-lattices $[L]$, the group $\operatorname{PGL}(3,F)$ acts well on $\mathcal{B}$ by left translation. We note that the stabilizer of the vertes $[\mathcal{O}^3]$ is the group $\operatorname{PGL}(3,\mathcal{O})$. Hence, we have a bijective correspondence 
$$[g\mathcal{O}]\longleftrightarrow gZ \operatorname{PGL}(3,\mathcal{O})$$ 
between the set of vertices of $\mathcal{B}$ and $\operatorname{PGL}(3,F)/\operatorname{PGL}(3,\mathcal{O})$. Additionally, the group $\operatorname{PGL}(3,F)$ acts isometrically on $\mathcal{B}$ since the left multiplication preserves the relation \eqref{eq:1.1}.

We define the \emph{color} of each vertex as follows. Let $\tau\colon\mathcal{B}\rightarrow \mathbb{Z}/3\mathbb{Z}$ be given by 
$$\tau([L]):=\log_q [\mathcal{O}^d:\pi^i L],$$
for a sufficiently large positive integer $i$ with $\pi^i L\subset \mathcal{O}^d.$
Since $[\pi^iL:\pi^{i+1}L]=d$, the color $\tau([L])$ is independent of the choice of the lattice in $[L]$ and hence is well-defined.

We denote by $K$ the group $\operatorname{PGL}(3,\mathcal{O})$.
\begin{lem}
For any vertex $x=gK\in G/K$ of color $i$ in $\mathcal{B}$, there are $q^2+q+1$ vertices of color $i+1$ neighbors of $gK$. They are given by
\[\left\{g\begin{pmatrix}\pi^{-1} & 0 & 0 \\ 0 & 1 & 0 \\ 0 & 0 & 1\end{pmatrix} K\right\}\bigcup\left\{g\begin{pmatrix}1 & b\pi^{-1} & 0 \\ 0 & \pi^{-1} & 0 \\ 0 & 0 & 1\end{pmatrix} K\colon b\in\mathcal{O}/\pi\mathcal{O}\right\}\]\[\bigcup\left\{g\begin{pmatrix}1 & 0 & c\pi^{-1} \\ 0 & 1 & d\pi^{-1} \\ 0 & 0 & \pi^{-1}\end{pmatrix} K\colon c,d\in\mathcal{O}/\pi\mathcal{O}\right\}.\]
Also, there are $q^2+q+1$ vertices of color $i+2$ neighbors of $gK$, which are 
\[\left\{g\begin{pmatrix}\pi^{-1} & 0 & 0 \\ 0 & \pi^{-1} & 0 \\ 0 & 0 & 1\end{pmatrix} K\right\}\bigcup\left\{g\begin{pmatrix}\pi^{-1} & 0 & 0 \\ 0 & 1 & b\pi^{-1} \\ 0 & 0 & \pi^{-1}\end{pmatrix} K\colon b\in\mathcal{O}/\pi\mathcal{O}\right\}\]\[\bigcup\left\{g\begin{pmatrix}1 & c\pi^{-1} & d\pi^{-1} \\ 0 & \pi^{-1} & 0 \\ 0 & 0 & \pi^{-1}\end{pmatrix} K\colon c,d\in\mathcal{O}/\pi\mathcal{O}\right\}.\]
\end{lem}

See Figure~\ref{fig:star} for the picture of vertices adjacent to a fixed vertex in $\mathcal{B}$. A similar discussion with the picture is presented in Subsection~3.2 of \cite{GP} and Section~3 of \cite{KL}.

\begin{figure}[H]
\begin{center}
\begin{tikzpicture}[scale=0.70]
\draw (-3,0)[fill] circle (3.3pt);
\draw [](0,0) circle (3.3pt);
\draw [fill] (2.88,-0.1) rectangle (3.08,0.1);
\draw [fill] (-3/2-0.1,+{sqrt(27)}/6-0.1) rectangle (-3/2+0.1,+{sqrt(27)}/6+0.1);
\draw [fill] (3/2,+{sqrt(27)}/2) circle (3.3pt);
\draw [fill] (-3/2-0.1,-{sqrt(27)}/2-0.1) rectangle (-3/2+0.1,-{sqrt(27)}/2+0.1);
\draw [fill] (3/2,-{sqrt(27)}/2) circle (3.3pt);
\draw [fill] (-3/2,-{sqrt(27)}/6) circle (3.3pt);
\draw [fill] (-3/2-0.1,+{sqrt(27)}/2-0.1) rectangle (-3/2+0.1,+{sqrt(27)}/2+0.1);
\draw [fill] (3/2-0.1,-{sqrt(3)}-0.1) rectangle (3/2+0.1,-{sqrt(3)}+0.1);
\draw [fill] (-0.1,+{sqrt(3)}-0.1) rectangle (0.1,+{sqrt(3)}+0.1);
\draw [fill] (-3/4,+{sqrt(2700)}/24) circle (3.3pt);
\draw [fill] (3/2-0.1,-{sqrt(3)}/2-0.1) rectangle (3/2+0.1,-{sqrt(3)}/2+0.1);
\draw [fill] (-3/4,-{sqrt(2700)}/24) circle (3.3pt); 
\draw [fill] (3/2-0.1,+{sqrt(3)}/2-0.1) rectangle (3/2+0.1,+{sqrt(3)}/2+0.1);

\draw [thick] (-3,0)--(-3/2,+{sqrt(27)}/2)--(3/2,+{sqrt(27)}/2)--(3,0)--(3/2,-{sqrt(27)}/2)--(-3/2,-{sqrt(27)}/2)--(-3,0);
\draw [thick] (-3,0)--(-3/2,+{sqrt(27)}/6);
\draw [thick] (-3/4,-{sqrt(2700)}/24)--(0,0)--(0,+{sqrt(27)}/3)--(-3/4,-{sqrt(2700)}/24); 
\draw [thick,shorten >=58pt] (-3/4,+{sqrt(2700)}/24)--(3/2,-{sqrt(27)}/6); 
\draw [thick,shorten >=28pt] (3/2,-{sqrt(27)}/6)--(-3/4,+{sqrt(2700)}/24); 
\draw [thick] (3,0)--(3/2,+{sqrt(27)}/6); 
\draw [thick] (-3/2,-{sqrt(27)}/6)--(-3/2,-{sqrt(27)}/2);  \draw[thick,shorten >=8pt] (-3/2,+{sqrt(27)}/6)--(0,0); \draw [thick] (-3/2,+{sqrt(27)}/6)--(-3/4,-{sqrt(2700)}/24);\draw [thick] (0,0)--(3/2,-{sqrt(3)})--(-3/4,-{sqrt(2700)}/24);\draw [thick] (-3/2,+{sqrt(27)}/2)--(-3/4,+{sqrt(2700)}/24); 
\draw[thick,shorten >=33pt] (3/2,-{sqrt(3)})--(-3/4,+{sqrt(2700)}/24);
\draw[thick,shorten >=72pt] (-3/4,+{sqrt(2700)}/24)--(3/2,-{sqrt(3)});\draw[thick,shorten >=43pt] (3/2,-{sqrt(27)}/2)--(0,0);\draw [thick] (3/2,-{sqrt(27)}/2)--(3/2,-{sqrt(3)});\draw [thick] (3/2,+{sqrt(27)}/2)--(0,+{sqrt(3)}); \draw[thick,shorten >=23pt] (-3/2,+{sqrt(27)}/2)--(0,0);\draw[thick, shorten >=25pt] (-3/4,+{sqrt(2700)}/24)--(0,0); \draw[thick, shorten >=20pt] (3/2,-{sqrt(3)}/2)--(0,0);
\draw[thick, shorten >=63pt] (-3/2,-{sqrt(3)}/2)--(2,-{sqrt(3)}/2);\draw[thick, shorten >=55pt] (-3/2,-{sqrt(3)}/2)--(3/2,-{sqrt(3)}/2);\draw[thick, shorten >=52pt] (3/2,-{sqrt(3)}/2)--(-3/2,-{sqrt(3)}/2); \draw[thick, shorten >=28] (-3,0)--(0,0);\draw[thick, shorten >=33] (3,0)--(0,0);\draw[thick](3/2,+{sqrt(3)}/2)--(3/2,-{sqrt(3)}/2);
\draw[thick, shorten >=15] (3/2,+{sqrt(3)}/2)--(0,0);
\draw[thick, shorten >=38] (-3/2,-{sqrt(27)}/2)--(0,0);
\draw [thick, shorten >=24] (3/2,+{sqrt(27)}/2)--(0,0);
\draw [thick,shorten >=52] (-3/2,-{sqrt(3)}/2)--(0,+{sqrt(3)});
\draw [thick,shorten >=31] (0,+{sqrt(3)})--(-3/2,-{sqrt(3)}/2);
\draw [thick,shorten >=36] (3/2,+{sqrt(3)}/2)--(-3/2,+{sqrt(3)}/2);
\draw [thick,shorten >=42] (-3/2,+{sqrt(3)}/2)--(3/2,+{sqrt(3)}/2);
\draw [thick,shorten >=33] (-3/2,-{sqrt(3)}/2)--(0,0);
%\node at (2,1.4) {\small{$q-1$}};\node at (-2,1.4) {\small{$q-1$}};\node at (2.2,-1.4) {\small{$(q-1)^2$}};\node at (-2.4,-1.2) {\small{$q-1$}};\node at (-0.9,-3.1) {\small{$(q-1)^2$}};\node at (2.3,-2.1) {\small{$q-1$}};\node at (0.4,2.2) {\small{$q-1$}};\node at (-0.6,3) {\small{$q-1$}};
%\draw[dashed] (-1,-{sqrt(299)}/6) -- (2,-{sqrt(12)}/3);
\end{tikzpicture}
\end{center}
\caption{Local neighborhood of a vertex in $\mathcal{B}$, $q=2$}\label{fig:star}
\end{figure}
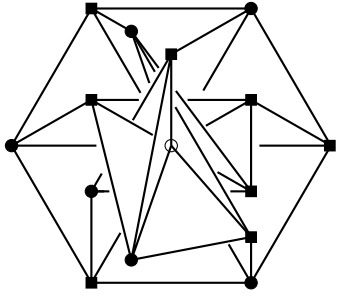

\subsection{Type of geodesics}

Our goal is to define the geometric zeta function $Z_\Gamma(q^{-s})$ of a discrete subgroup $\Gamma$ of $\operatorname{PGL}(3,F)$. For this, we need to decide what kind of closed geodesics to count. The $2$-dimensional building $\mathcal{B}$ is obtained by gluing together its apartments along chambers. Each apartment is a Euclidean plane tiled by equilateral triangles. The Euclidean metric on each apartment in turn yields the metric on the building $\mathcal{B}$. 

We note that adjacent vertices have different colors. The \emph{type} $\tau(x\to y)$ of a directed edge $x\to y$ from a vertex $x$ to a vertex $y$ is defined to be $\tau(y)-\tau(x)$. If $\widetilde{e}=x\to y$ is a directed edge from $x$ to $y$ in $\mathcal{B}$, then we denote by $s(\widetilde{e})=x$ (source) and $t(\widetilde{e})=y$ (target). A sequence of $\widetilde{e_1},\widetilde{e_2},\ldots,\widetilde{e_n}$ of directed edges in $\mathcal{B}$ is called a path if $t(\widetilde{e_k})=s(\widetilde{e_{k+1}})$ for all $1\le k\le n-1$. If it consists of type $i$ directed edges, then it is called a path of type $i$.

\begin{figure}[H]
\begin{center}
\begin{tikzpicture}[scale=0.8]
\draw[-{Stealth[open, length=3mm, width=2mm]}] (-7,0)--(-5,0);
\draw[-{Stealth[open, length=3mm, width=2mm]}] (-7,0)--(-8,{sqrt(3)});
\draw[-{Stealth[open, length=3mm, width=2mm]}] (-7,0)--(-8,-{sqrt(3)});
\draw[-{Stealth[length=3mm, width=2mm]}] (0,0)--(-2,0);
\draw[-{Stealth[length=3mm, width=2mm]}] (0,0)--(1,{sqrt(3)});
\draw[-{Stealth[length=3mm, width=2mm]}] (0,0)--(1,-{sqrt(3)});
\fill (-7,0) circle (3pt);\fill (0,0) circle (3pt);

\draw [-{Stealth[open, length=3mm, width=2mm]}] (-7,-2) -- (-5,-2);
\node [draw] at (-6,-1.3) {Type 1};
\draw [-{Stealth[length=3mm, width=2mm]}] (-2,-2) -- (0,-2);
\node [draw] at (-1,-1.3) {Type 2};
\end{tikzpicture}
\end{center}
\caption{Type of directed edges}\label{fig:type}
\end{figure}
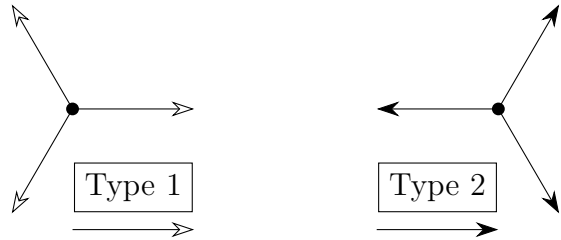

A path $\widetilde{e_1},\widetilde{e_2},\ldots,\widetilde{e_n}$ in $\mathcal{B}$ is called a \emph{geodesic path} if it is a part of straight line in an apartment in $\mathcal{B}$. Equivalently, it is a path with the condition that $s(\widetilde{e_k}),s(\widetilde{e_{k+1}})=t(\widetilde{e_k}),t(\widetilde{e_{k+1}})$ do not form a chamber in $\mathcal{B}$ for all $1\le k\le n-1$. See Figure~\ref{fig:geopath}.

\begin{center}
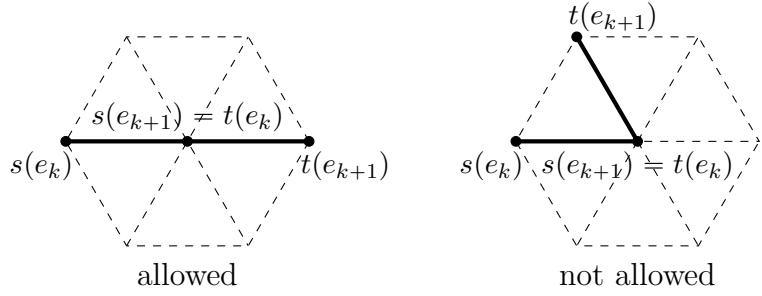
\begin{figure}[H]
\label{fig:geopath}
\newcommand*\rows{6}
\begin{tikzpicture}[scale=1.6]

\draw[dashed] (0:1) --++(120:1) --++(180:1) --++(240:1) --++ (300:1) --++(0:1) --++ (60:1);
\draw[dashed] (180:1) --++ (0:2);
\draw[dashed] (240:1) --++ (60:2);
\draw[dashed] (120:1) --++ (300:2);

\draw[ultra thick] (180:1) --++ (0:2);
\draw[ultra thick] (2.7,0) --++ (0:1) --++ (120:1);

\draw[dashed] (2.7,0) --++(0:2) --++(120:1) --++ (180:1) --++(240:1) --++ (300:1) --++(0:1) --++ (60:1);
\draw[dashed] (3.7,0) --++ (60:1);
\draw[dashed] (3.7,0) --++ (120:1);
\draw[dashed] (3.7,0) --++ (240:1);
\draw[dashed] (3.7,0) --++ (300:1);

\draw[fill] (0,0) circle (0.04cm);
\draw[fill] (-1,0) circle (0.04cm);
\draw[fill] (1,0) circle (0.04cm);
\draw[fill] (2.7,0) circle (0.04cm);
\draw[fill] (3.7,0) circle (0.04cm);
\draw[fill] (3.2,{sqrt(3)/2}) circle (0.04cm);

\node (a) at (270:1.1) {allowed};
\node (b) at (3.7,-1.1) {not allowed};
\node (c) at (-1.2,-0.2) {\small$s(e_k)$};
\node (d) at (0,0.2) {\small$s(e_{k+1})=t(e_k)$};
\node (e) at (1.3,-0.2) {\small$t(e_{k+1})$};
\node (f) at (2.5,-0.2) {\small$s(e_k)$};
\node (g) at (3.7,-0.2) {\small$s(e_{k+1})=t(e_k)$};
\node (h) at (3.5,1.03) {\small$t(e_{k+1})$};
\end{tikzpicture}
\caption{Type $1$ geodesic paths, allowed one and what is not }\label{fig:geopath}
\end{figure}
\end{center}

Let $a=\operatorname{diag}(t,1,1)\in\operatorname{PGL}(3,\mathbb{F}_q(\!(t^{-1})\!))$. Then, the group $H$ given by
\[H=K\cap aKa^{-1}=\left\{\begin{pmatrix}k_{11} & k_{12} & k_{13} \\ k_{21} & k_{22} & k_{23} \\ k_{31} & k_{32} & k_{33}\end{pmatrix}\in K\colon k_{21},k_{31}\in t^{-1}\mathbb{F}_q[\![t^{-1}]\!]\right\},\] called the \emph{parahoric subgroup} of $\operatorname{PGL}(3,\mathbb{F}_q(\!(t^{-1})\!))$, is the stabilizer of the type 1 directed edge $K\to aK$ in $\mathcal{B}$. Thus, we may identify the set of type 1 directed edges with $G/H$.

We observe the following.

\begin{lem}
For each positive integer $n$ and a finite segment $(\widetilde{e_0},\ldots,\widetilde{e_{n-1}})\in (G/H)^{n}$ of type 1 geodesic in $\mathcal{B}$, 
there are $q^2$ distinct $\widetilde{e_n}\in G/H$ such that $(\widetilde{e_0},\ldots,\widetilde{e_{n-1}},\widetilde{e_n})\in (G/H)^{n+1}$ is also a type 1 geodesic segment.
\end{lem}
\begin{proof}
This is a special case of Lemma~2.1 in \cite{CM}.
\end{proof}

%%----
%% Section 3
%%----

\section{Admissible paths in non-uniform quotient}\label{sec:3}

In this section, we define and classify admissible paths in the standard arithmetic quotient of the building. We first introduce the geometric description of the standard arithmetic quotient of $\operatorname{PGL}_3$ over a positive characteristic local field and explore which paths are considered relevant for the zeta function.

\subsection{Fundamental domain and weight}

Let $\mathbb{F}_q$ be the finite field of order $q$. Denote by $\mathbb{F}_q[t]$ and $\mathbb{F}_q(t)$ the ring of polynomials and the field of rational functions over $\mathbb{F}_q$, respectively. The absolute value $\|\cdot\|$ of $\mathbb{F}_q(t)$ is defined for any $f\in \mathbb{F}_q(t)$, by 
$$\|{f}\|:=q^{\deg (g)-\deg (h)},$$
where $g,h$ are polynomials over $\mathbb{F}_q$ with $f=\frac{g}{h}$.
The completion of $\mathbb{F}_q(t)$ with respect to $\|\cdot\|$, the field of formal Laurent series in $t^{-1}$, is denoted by $\mathbb{F}_q(\!(t^{-1})\!)$, i.e.,
$$\mathbb{F}_q(\!(t^{-1})\!)=\left\{\sum_{n=-N}^\infty a_nt^{-n}:N\in \mathbb{Z}, a_n\in \mathbb{F}_q\right\}.$$
%Let $\mathbf{Z}$ be the subring $\mathbb{F}_q[t]$ of polyonials in $t$ over $\mathbb{F}_q$. 
The valuation ring $\mathcal{O}$ is the subring of power series 
$$\mathbb{F}_q\mathbb{[\![}t^{-1}]\!]=\left\{\sum_{n=0}^\infty a_nt^{-n}: a_n\in \mathbb{F}_q\right\}.$$ We note that $\operatorname{PGL}(3,\mathbb{F}_q[t])$ is a discrete subgroup of $\operatorname{PGL}(3,\mathbb{F}_q(\!(t^{-1})\!))$.

\begin{lem}
Given every $g\in \operatorname{PGL}(3,\mathbb{F}_q(\!(t^{-1})\!)$, there exists a unique pair of non-negative integers $(m,n)$ with $0\le n\le m$ such that
$$g=\gamma a_{m,n}w\in \operatorname{PGL}(3,\mathbb{F}_q[t])\begin{pmatrix}t^m & 0 & 0 \\ 0 & t^n & 0 \\ 0 & 0 & 1\end{pmatrix} \operatorname{PGL}(3,\mathcal{O}).$$
\end{lem}
\begin{proof}
This follows from the reduction theory of reductive groups over rational function fields discussed in \cite{Pr}. See also Lemma~3.2 of \cite{HK}.
\end{proof}

For a pair of non-negative integers $(m,n)$ with $m\ge n\ge 0$, let $v_{m,n}$ be the vertex of the quotient complex $\operatorname{PGL}(3,\mathbb{F}_q[t])\backslash \mathcal{B}$ corresponds to the double coset
$$\operatorname{PGL}(3,\mathbb{F}_q[t]) \begin{pmatrix} t^m & 0 & 0 \\ 0 & t^n & 0 \\ 0 & 0 & 1\end{pmatrix}\operatorname{PGL}(3,\mathcal{O}).$$

The definition of the building $\mathcal{B}$ implies that there is an edge between two vertices $v_{m,n}$ and $v_{m',n'}$ if and only if the following hold:
\begin{displaymath}
\begin{cases} (m',n')\in\{(m\pm1,n),(m,n\pm1),(m\pm1,n\pm1)\} &\text{ if }m>n>0\\
(m',n')\in \{(m\pm1,n),(m,n+1),(m+1,n+1)\} &\text{ if }m>n=0\\
(m',n')\in \{(m+1,n),(m,n-1),(m\pm1,n\pm1)\} &\text{ if }m=n>0\\
(m',n')\in \{(1,0),(1,1)\}&\text{ if }m=n=0.
\end{cases}
\end{displaymath}
We conclude that the quotient complex $\Gamma\backslash\mathcal{B}$ is described in Figure~\ref{fig:funddom}.
%\begin{figure}[h]\begin{center}{\includegraphics[width=90mm]{fundametal.png}}\put(-240,5){$v_{0,0}$}\put(-185,5){$v_{1,0}$}\put(-220,61){$v_{1,1}$}\put(-130,5){$v_{2,0}$}\put(-80,5){$v_{3,0}$}\put(-30,5){$v_{4,0}$}\put(-170,61){$v_{2,1}$}\put(-118,61){$v_{3,1}$}\put(-65,61){$v_{4,1}$}\put(-194,105){$v_{2,2}$}\put(-145,105){$v_{3,2}$}\put(-92,105){$v_{4,2}$}\put(-40,105){$v_{5,2}$}\put(-168,150){$v_{3,3}$}\put(-118,150){$v_{4,3}$}\put(-66,150){$v_{5,3}$}\put(-142,195){$v_{4,4}$}\put(-92,195){$v_{5,4}$}\put(-42,195){$v_{6,4}$}\end{center}\caption{The quotient complex $\Gamma\backslash\mathcal{B}(G)$}\label{figure4}\end{figure}

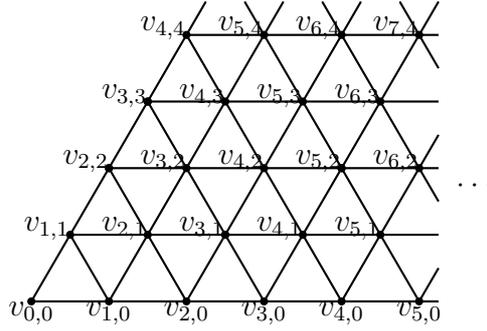
\begin{figure}[H]
\begin{center}
\begin{tikzpicture}[scale=0.51]
  \draw [thick](-5,0) -- (5.5,0);   \draw [thick](-0.5,+{sqrt(60.75)}) -- (-5,0);\draw[thick](-4,+{sqrt(3)}) -- (-3,0);\draw[thick] (-3,+{sqrt(12)}) -- (-1,0);\draw [thick](-2,{sqrt(27)}) -- (1,0);\draw[thick] (-1,{sqrt(48)})--(3,0);\draw[thick] (-4,+{sqrt(3})--(5.5,+{sqrt(3)});\draw[thick](-3,+{sqrt(12)})--(5.5,+{sqrt(12)});\draw[thick](-2,+{sqrt(27)})--(5.5,+{sqrt(27)});\draw[thick](-1,+{sqrt(48)})--(-1,+{sqrt(48)});\draw[thick](1.5,+{sqrt(60.75)})--(-3,0);\draw[thick](3.5,+{sqrt(60.75)})--(-1,0);\draw[thick](5.5,+{sqrt(60.75)})--(1,0);\draw[thick](5.5,+{sqrt(75)}/2)--(3,0);\draw[thick](5,0)--(0.5,+{sqrt(60.75)});\draw[thick] (-1,+{sqrt(48)})--(5.5,+{sqrt(48)});\draw[thick](5.5,+{sqrt(27)}/2)--(2.5,+{sqrt(60.75)});\draw[thick] (4.5,+{sqrt(60.75)})--(5.5,+{sqrt(147)}/2);\draw[thick] (5,0)--(5.5,+{sqrt(3)}/2);
\node at (-5,-0.3) {$v_{0,0}$};\node at (-3,-0.3) {$v_{1,0}$};\node at (-1,-0.3) {$v_{2,0}$};\node at (1,-0.3) {$v_{3,0}$};\node at(3,-0.3) {$v_{4,0}$};\node at (5,-0.3) {$v_{5,0}$};\node at (6.5,3) {$\cdots$};\node at (-4.6,1.9) {$v_{1,1}$};\node at (-2.6,1.9) {$v_{2,1}$};\node at (-0.6,1.9) {$v_{3,1}$};\node at (1.4,1.9) {$v_{4,1}$};\node at (3.4,1.9) {$v_{5,1}$};\node at (-3.6,3.65) {$v_{2,2}$};\node at (-1.6,3.65) {$v_{3,2}$};\node at (0.4,3.65) {$v_{4,2}$};\node at(2.4,3.65) {$v_{5,2}$};\node at (4.4,3.65) {$v_{6,2}$};\node at (-2.6,5.35) {$v_{3,3}$};\node at (-0.6,5.35) {$v_{4,3}$};\node at (1.4,5.35) {$v_{5,3}$};\node at (3.4, 5.35) {$v_{6,3}$};\node at (-1.6, 7.1) {$v_{4,4}$};\node at (0.4, 7.1) {$v_{5,4}$};\node at (2.4, 7.1) {$v_{6,4}$};\node at (4.4, 7.1) {$v_{7,4}$};
\fill (-5,0)    circle (3pt); \fill (-3,0)    circle (3pt); \fill (-1,-0)    circle (3pt);\fill (1,-0)    circle (3pt); \fill (3,0) circle (3pt);\fill(5,0) circle (3pt);\fill (-4,+{sqrt(3)})    circle (3pt); \fill (-3,+{sqrt(12)})    circle (3pt); \fill (-2,{sqrt(27)})    circle (3pt);\fill (-1,+{sqrt(48)})    circle (3pt); \fill (-2,+{sqrt(3)}) circle (3pt);\fill(-0,+{sqrt(3)}) circle (3pt);\fill (0,+{sqrt(3)}) circle (3pt);\fill(2,{sqrt(3)}) circle (3pt);\fill (4,+{sqrt(3)}) circle (3pt);\fill (-3,{sqrt(12)}) circle (3pt); \fill (-1,{sqrt(12)}) circle (3pt);\fill (1,{sqrt(12)}) circle (3pt);\fill(3,{sqrt(12)}) circle (3pt);\fill (3,{sqrt(12)}) circle (3pt);\fill(5,{sqrt(12)}) circle (3pt);\fill (-2,{sqrt(27)}) circle (3pt);\fill (0,{sqrt(27)}) circle (3pt);\fill (2,{sqrt(27)}) circle (3pt); \fill (4,{sqrt(27)}) circle (3pt);\fill(1,{sqrt(48)}) circle (3pt); \fill(3,{sqrt(48)}) circle (3pt);\fill (5,{sqrt(48)}) circle (3pt);
\end{tikzpicture}
\end{center}
\caption{The fundamental domain for $\Gamma\backslash\mathcal{B}$}\label{fig:funddom}
\end{figure}

\subsection{Paths in the quotient that lift to geodesics}
Our goal is to count the number of cycles that are obtained as quotients of geodesics in the universal covering. Let $\pi\colon \mathcal{B}\to \operatorname{PGL}(3,\mathbb{F}_q[t])\backslash\mathcal{B}$ be the natural projection map. A sequence of $e_1,e_2,\ldots,e_n$ of directed edges in $\operatorname{PGL}(3,\mathbb{F}_q[t])\backslash\mathcal{B}$ is called a path if $t(e_k)=s(e_{k+1})$ for all $1\le k\le n-1$. Then, this is equivalent to being represented as the image of a path in the building $\mathcal{B}$ under $\pi$. 

%A path $e_1,e_2,\ldots,e_n$ in $\mathcal{B}$ is called a \emph{geodesic path} if it is a part of straight line in an apartment in $\mathcal{B}$. Equivalently, it is a path with the condition that $s(e_k),s(e_{k+1})=t(e_k),t(e_{k+1})$ do not form a chamber in $\mathcal{B}$ for all $1\le k\le n-1$. 
We will call a path in $\Gamma\backslash\mathcal{B}$ \emph{admissible} if it appears as the projection of a geodesic. Since $\operatorname{PGL}(3,\mathbb{F}_q[t])$-action on $\mathcal{B}$ preserves the type of directed edges, we can define similarly the type of admissible paths in $\operatorname{PGL}(3,\mathbb{F}_q[t])\backslash\mathcal{B}$. If it consists of type $i$ directed edges, then it is called an admissible path of type $i$.
 
For a directed edge $\widetilde{e}$ of type $1$ in $\mathcal{B}$ and a directed edge $e'$ of type $1$ in $X=\operatorname{PGL}(3,\mathbb{F}_q[t])\backslash\mathcal{B}$, let
$$w(\widetilde{e},e')=\#\{\widetilde{e}'\in \operatorname{E}_1(\mathcal{B})\colon \pi(\widetilde{e}')=e',(\widetilde{e},\widetilde{e}') \textrm{ is a geodesic path in }\mathcal{B}\}.$$ It counts the number of type $1$ edges $\widetilde{e}'$ in $\mathcal{B}$ that does not form a chamber with the other one $\widetilde{e}$ in $\mathcal{B}$ and projects to a fixed one $e'$ in $X$. For an edge $e$ of $X=\Gamma\backslash \mathcal{B}$, we write $w(e,e')=w(\widetilde{e},e')$, where $\widetilde{e}$ is any preimage of $e$ in $\mathcal{B}$. This is again well-defined since $w(e,e')$ does not depend on the specific choice of the preimage $\widetilde{e}$. In particular, $w(e,e')\ne0$ if and only if $(e,e')$ is an admissible path in $\operatorname{PGL}(3,\mathbb{F}_q[t])\backslash\mathcal{B}$. %and it counts the number of $\Gamma$-equivalence classes of pair $(\widetilde{e},\widetilde{e}')$ of edges in $\mathcal{B}$ such that $s(\widetilde{e}),s(\widetilde{e}'),t(\widetilde{e}')$ do not form a chamber in $\mathcal{B}$.

Now for a closed path $p=(e_1,e_2,\ldots,e_n)$, let 
$$w(p)=\prod_{j\textrm{ mod }n}w(e_j,e_{j+1}).$$ 

We examined examples of type 1 and type 2 edges in Figure~\ref{fig:type}. In the fundamental domain, these edges can be classified into three families based on their slopes. By examining the projection behavior under the quotient by $\Gamma$ among all vertices linked to a fixed vertex in the building (see Figure~\ref{fig:local-general}), we derive the local pattern of the weights $w(e,e')$. Figure~\ref{fig:localpattern} illustrates the various cases of $w(e,e')$ for different pairs of $e$ and $e'$. These probabilities depend on the direction in which they approach a fixed vertex. In cases where they cross the boundary of the quotient complex, they will be considered as entering the reflected vertex across the boundary. 

 %In this section, we discuss the conditions for admissibility and explain the probability of transitioning to the next admissible path for each edge. 

%Thus, we obtain the local transition probabilities centered at $v_{m,n}$ of type $1$ and type $2$ geodesic flow on $\Gamma\backslash G/K$ in Figure~\ref{figure4}. 

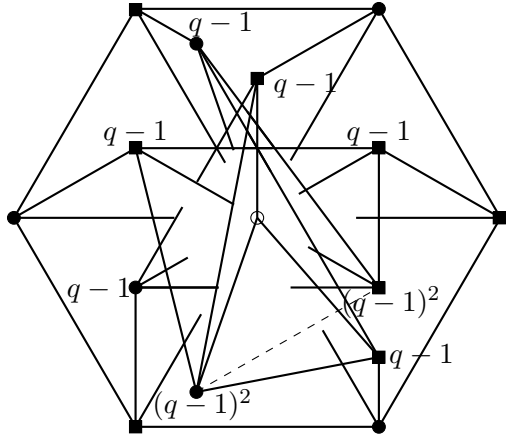
\begin{figure}[H]
\begin{center}
\begin{tikzpicture}[scale=0.8]
\draw (-4,0)[fill] circle (3pt);
\draw [](0,0) circle (3pt);
\draw [fill] (3.88,-0.1) rectangle (4.08,0.1);
\draw [fill] (-1.9,3.35) rectangle (-2.1,3.55);
\draw [fill] (2,{sqrt(12)}) circle (3pt);
\draw [fill] (-1.9,-{sqrt(12)}-0.1) rectangle (-2.1,-{sqrt(12)}+0.1);
\draw [fill] (2,-{sqrt(12)}) circle (3pt);
\draw [fill] (-2,-{sqrt(12)}/3) circle (3pt);
\draw [fill] (-2.1,1.06) rectangle (-1.9,1.26);
\draw [fill] (1.9,1.06) rectangle (2.1,1.26);
\draw [fill] (2.1,-1.06) rectangle (1.9,-1.26);
\draw [fill] (-1,+{sqrt(300)}/6) circle (3pt);
\draw [fill] (0.1,2.41) rectangle (-0.1,2.21);
\draw [fill] (-1,-{sqrt(300)}/6) circle (3pt); 
\draw [fill] (2.1,-2.41) rectangle (1.9,-2.21);

\draw [thick] (-4,0)--(-2,{sqrt(12)})--(2,{sqrt(12)})--(4,0)--(2,-{sqrt(12)})--(-2,-{sqrt(12)})--(-4,0);
\draw [thick] (-4,0)--(-2,+{sqrt(12)}/3);\draw [thick] (-1,-{sqrt(300)}/6)--(0,0)--(0,+{sqrt(48)}/3)--(-1,-{sqrt(300)}/6); 
\draw [thick,shorten >=66pt] (-1,+{sqrt(300)}/6)--(2,-{sqrt(12)}/3); 
\draw [thick,shorten >=32pt] (2,-{sqrt(12)}/3)--(-1,+{sqrt(300)}/6); 
\draw [thick] (4,0)--(2,+{sqrt(12)}/3); \draw [thick] (-2,-{sqrt(12)}/3)--(-2,-{sqrt(12)});  
\draw[thick,shorten >=10pt] (-2,+{sqrt(12)}/3)--(0,0); \draw [thick] (-2,+{sqrt(12)}/3)--(-1,-{sqrt(300)}/6);\draw [thick] (0,0)--(2,-{sqrt(48)}/3)--(-1,-{sqrt(300)}/6);\draw [thick] (-2,+{sqrt(12)})--(-1,+{sqrt(300)}/6); 
\draw[thick,shorten >=37pt] (2,-{sqrt(48)}/3)--(-1,+{sqrt(300)}/6);
\draw[thick,shorten >=81pt] (-1,+{sqrt(300)}/6)--(2,-{sqrt(48)}/3);
\draw[thick,shorten >=49pt] (2,-{sqrt(12)})--(0,0);\draw [thick] (2,-{sqrt(12)})--(2,-{sqrt(48)}/3);\draw [thick] (2,+{sqrt(12)})--(0,+{sqrt(48)}/3); \draw[thick,shorten >=24pt] (-2,+{sqrt(12)})--(0,0);\draw[thick, shorten >=27pt] (-1,+{sqrt(300)}/6)--(0,0); \draw[thick, shorten >=22pt] (2,-{sqrt(12)}/3)--(0,0);
\draw[thick, shorten >=60pt] (-2,-{sqrt(12)}/3)--(2,-{sqrt(12)}/3);\draw[thick, shorten >=60pt] (-2,-{sqrt(12)}/3)--(2,-{sqrt(12)}/3);\draw[thick, shorten >=58pt] (2,-{sqrt(12)}/3)--(-2,-{sqrt(12)}/3); \draw[thick, shorten >=31] (-4,0)--(0,0);\draw[thick, shorten >=37] (4,0)--(0,0);\draw[thick](2,+{sqrt(12)}/3)--(2,-{sqrt(12)}/3);\draw[thick, shorten >=18] (2,+{sqrt(12)}/3)--(0,0);\draw[thick, shorten >=42] (-2,-{sqrt(12)})--(0,0);\draw [thick, shorten >=25] (2,+{sqrt(12)})--(0,0);\draw [thick,shorten >=56] (-2,-{sqrt(12)}/3)--(0,+{sqrt(48)}/3);\draw [thick,shorten >=46] (0,+{sqrt(48)}/3)--(-2,-{sqrt(12)}/3);\draw [thick,shorten >=40] (2,+{sqrt(12)}/3)--(-2,+{sqrt(12)}/3);\draw [thick,shorten >=40] (-2,+{sqrt(12)}/3)--(2,+{sqrt(12)}/3);\draw [thick,shorten >=30] (-2,-{sqrt(12)}/3)--(0,0);
\node at (2,1.4) {\small{$q-1$}};\node at (-2,1.4) {\small{$q-1$}};\node at (2.2,-1.4) {\small{$(q-1)^2$}};\node at (-2.6,-1.2) {\small{$q-1$}};\node at (-0.9,-3.1) {\small{$(q-1)^2$}};\node at (2.7,-2.3) {\small{$q-1$}};\node at (0.8,2.2) {\small{$q-1$}};\node at (-0.6,3.2) {\small{$q-1$}};

\draw[dashed] (-1,-{sqrt(299)}/6) -- (2,-{sqrt(12)}/3);
\end{tikzpicture}
\end{center}
\caption{Local neighborhood of a vertex in $\mathcal{B}$}\label{fig:local-general}
\end{figure}

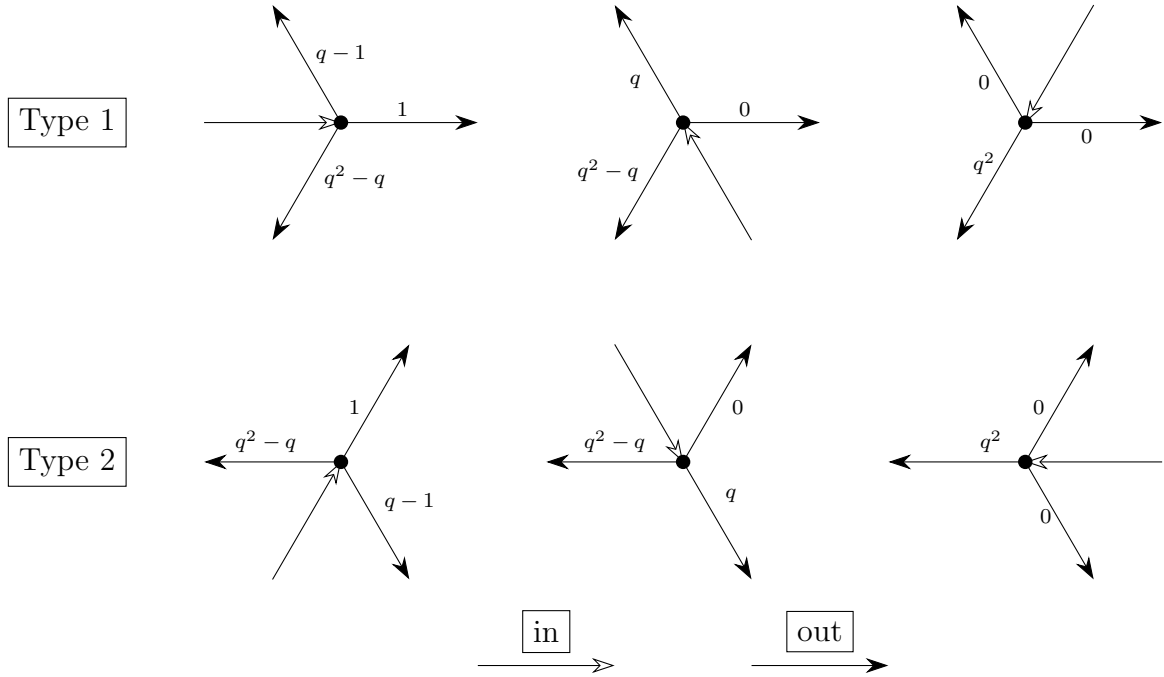
\begin{figure}[H]
\begin{center}
\begin{tikzpicture}[scale=0.9]
\draw[-{Stealth[open, length=3mm, width=2mm]}] (-7,0)--(-5,0);
\draw[-{Stealth[length=3mm, width=2mm]}] (-5,0)--(-3,0);
\draw[-{Stealth[length=3mm, width=2mm]}] (-5,0)--(-6,{sqrt(3)});
\draw[-{Stealth[length=3mm, width=2mm]}] (-5,0)--(-6,-{sqrt(3)});
\draw[-{Stealth[open, length=3mm, width=2mm]}] (1,-{sqrt(3)})--(-0,0);
\draw[-{Stealth[length=3mm, width=2mm]}] (0,0)--(2,0);
\draw[-{Stealth[length=3mm, width=2mm]}] (0,0)--(-1,{sqrt(3)});
\draw[-{Stealth[length=3mm, width=2mm]}] (0,0)--(-1,-{sqrt(3)});
\draw[-{Stealth[open, length=3mm, width=2mm]}] (6,{sqrt(3)})--(5,0);
\draw[-{Stealth[length=3mm, width=2mm]}] (5,0)--(7,0);
\draw[-{Stealth[length=3mm, width=2mm]}] (5,0)--(4,{sqrt(3)});
\draw[-{Stealth[length=3mm, width=2mm]}] (5,0)--(4,-{sqrt(3)});
\fill (-5,0) circle (3pt);\fill (0,0) circle (3pt);\fill (5,0) circle (3pt);
\node at (-4.1,0.2) {\tiny{$1$}};\node at (-5,1) {\tiny{$q-1$}};\node at (-4.8,-0.8){\tiny{$q^2-q$}};\node at(-1.1,-0.7) {\tiny{$q^2-q$}};\node at (-0.7,0.6) {\tiny{$q$}};\node at (0.9,0.2) {\tiny{$0$}}; \node at (4.4,0.6) {\tiny{$0$}}; \node at (5.9,-0.2) {\tiny{$0$}};\node at (4.4,-0.6) {\tiny{$q^2$}};
\draw[-{Stealth[open, length=3mm, width=2mm]}] (-6,-{sqrt(3)}-5)--(-5,-5);
\draw[-{Stealth[length=3mm, width=2mm]}] (-5,-5)--(-7,-5);
\draw[-{Stealth[length=3mm, width=2mm]}] (-5,-5)--(-4,+{sqrt(3)}-5);
\draw[-{Stealth[length=3mm, width=2mm]}] (-5,-5)--(-4,-{sqrt(3)}-5);
\draw[-{Stealth[open, length=3mm, width=2mm]}] (-1,+{sqrt(3)}-5)--(-0,-5);
\draw[-{Stealth[length=3mm, width=2mm]}] (0,-5)--(-2,-5);
\draw[-{Stealth[length=3mm, width=2mm]}] (0,-5)--(1,+{sqrt(3)}-5);
\draw[-{Stealth[length=3mm, width=2mm]}] (0,-5)--(1,-{sqrt(3)}-5);
\draw[-{Stealth[open, length=3mm, width=2mm]}] (7,-5)--(5,-5);
\draw[-{Stealth[length=3mm, width=2mm]}] (5,-5)--(3,-5);
\draw[-{Stealth[length=3mm, width=2mm]}] (5,-5)--(6,+{sqrt(3)}-5);
\draw[-{Stealth[length=3mm, width=2mm]}] (5,-5)--(6,-{sqrt(3)}-5);
\draw [-{Stealth[open, length=3mm, width=2mm]}] (-3,-8) -- (-1,-8);
\node [draw] at (-2,-7.5) {in};
\node [draw] at (-9,0) {Type 1};
\node [draw] at (-9,-5) {Type 2};
\draw [-{Stealth[length=3mm, width=2mm]}] (1,-8) -- (3,-8);
\node [draw] at (2,-7.5) {out};
\fill (-5,-5) circle (3pt);\fill (0,-5) circle (3pt);\fill (5,-5) circle (3pt);
\node at (-4.8,-4.2) {\tiny{$1$}}; \node at (-4,-5.6) {\tiny{$q-1$}};\node at (-6.1,-4.7) {\tiny{$q^2-q$}};
\node at (0.8,-4.2) {\tiny{$0$}}; \node at (-1,-4.7) {\tiny{$q^2-q$}};\node at (0.7,-5.5) {\tiny{$q$}};
\node at (5.2,-4.2) {\tiny{$0$}}; \node at (5.3,-5.8) {\tiny{$0$}};\node at (4.5,-4.7) {\tiny{$q^2$}};
\end{tikzpicture}
\end{center}
\caption{Local pattern of the weights $w(e,e')$ for each type}\label{fig:localpattern}
\end{figure}
%\begin{figure}[H]
%\begin{center}
%\begin{tikzpicture}[scale=0.8]
%\draw[gray,thick,shorten >=3pt,->] (-4,0)--(-6,0);
%\draw[thick,shorten >=3pt,->] (-6,0)--(-8,0);
%\draw[thick,shorten >=3pt,->] (-6,0)--(-5,{sqrt(3)});
%\draw[thick,shorten >=3pt,->] (-6,0)--(-5,-{sqrt(3)});
%\draw[gray, thick, shorten >=3pt,->] (-1,+{sqrt(3)})--(-0,0);\draw[thick,shorten >=3pt,->] (0,0)--(-2,0);
%\draw[thick,shorten >=3pt,->] (0,0)--(1,-{sqrt(3)});
%\draw[thick,shorten >=3pt,->] (0,0)--(1,+{sqrt(3)});
%\draw[gray, thick, shorten >=3pt,->] (5,-{sqrt(3)})--(6,0);\draw[thick,shorten >=3pt,->] (6,0)--(4,0);
%\draw[thick,shorten >=3pt,->] (6,0)--(7,{sqrt(3)});
%\draw[thick,shorten >=3pt,->] (6,0)--(7,-{sqrt(3)});
%\fill (-6,0) circle (3pt);\fill (0,0) circle (3pt);\fill (6,0) circle (3pt);\node at (-6.8,0.3) {\tiny{$q^2$}};\node at (-5.8,1) {\tiny{$0$}};\node at (-5.8,-0.8){\tiny{$0$}};\node at(-1,0.3) {\tiny{$q^2-q$}};\node at (0.6,-0.5) {\tiny{$q$}};\node at (0.7,0.5) {\tiny{$0$}}; \node at (6.2,0.8) {\tiny{$1$}}; \node at (7,-0.6) {\tiny{$q-1$}};\node at (5.1,0.3) {\tiny{$q^2-q$}};
%\end{tikzpicture}
%\end{center}
%\caption{Local transition probability of type 1 and type 2 edges}\label{figure5}
%\end{figure}

Now, by treating each edge as an alphabet and drawing an arrow from $e$ to $e'$ when $w(e,e')$ is not zero, we create a large directed graph. This process is similar to representing a shift space as a directed graph. Figure~\ref{fig:descripD} describes the graph where the arrows are drawn differently depending on the value of $w(e,e')$.

\begin{center}
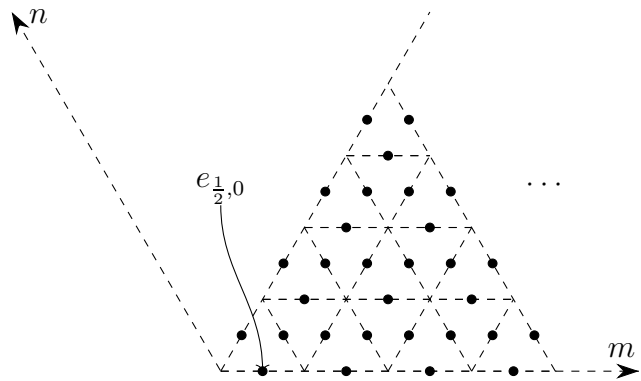
\begin{figure}[h]
\newcommand*\rows{4}
\begin{tikzpicture}[scale=1.1]
    \foreach \row in {0, 1, ...,\rows} {
        \draw[dashed] ($\row*(0.5, {0.5*sqrt(3)})$) -- ($(\rows,0)+\row*(-0.5, {0.5*sqrt(3)})$);
        \draw[dashed] ($\row*(1, 0)$) -- ($(\rows/2,{\rows/2*sqrt(3)})+\row*(0.5,{-0.5*sqrt(3)})$);
        \draw[dashed] ($\row*(1, 0)$) -- ($(0,0)+\row*(0.5,{0.5*sqrt(3)})$);
    }
    \foreach \row in {0, 1, ...,3} {
        \draw[dashed] ($\row*(0.5, {0.5*sqrt(3)})$) -- ($(\rows,0)+\row*(-0.5, {0.5*sqrt(3)})$);
        \draw[fill] (0.5+\row*1, 0) circle (1.5pt);
    }
    \foreach \row in {0, 1, ..., 3}{  
        \draw[fill] (0.25+\row,{sqrt(3)/4}) circle (1.5pt);
        \draw[fill] (0.75+\row,{sqrt(3)/4}) circle (1.5pt);
    }
    \foreach \row in {0, 1, ..., 2}{
        \draw[fill] (1+\row,{sqrt(3)/2}) circle (1.5pt);
        \draw[fill] (0.75+\row,{3*sqrt(3)/4}) circle (1.5pt);
        \draw[fill] (1.25+\row,{3*sqrt(3)/4}) circle (1.5pt);
    }
    \foreach \row in {0, 1}{
        \draw[fill] (1.5+\row,{sqrt(3)}) circle (1.5pt);
        \draw[fill] (1.25+\row,{5*sqrt(3)/4}) circle (1.5pt);
        \draw[fill] (1.75+\row,{5*sqrt(3)/4}) circle (1.5pt);
    }
    \foreach \row in {0}{
        \draw[fill] (2+\row,{3*sqrt(3)/2}) circle (1.5pt);
        \draw[fill] (1.75+\row,{7*sqrt(3)/4}) circle (1.5pt);
        \draw[fill] (2.25+\row,{7*sqrt(3)/4}) circle (1.5pt);
    }

\draw[dashed] (60:4) --++(60:1);
\draw[dashed] (0:4) --++(0:1);
%\draw[fill] (30:3.466) circle (0.05);
\node (a) at (30:4.5) {$\cdots$};
\node (b) at (3:4.8) {$m$};
\node (c) at (117:4.8) {$n$};
%\node (d) at (39:6.3) {$v_{4,2}$};
\node (e) at (0,2.2) {$e_{\frac{1}{2},0}$};
%\draw[->](38:6.2) to [out=270,in=90] node[above]{}(30:3.466);
\draw[->](0,2) to [out=270,in=90] node[above]{}(0:0.5);

\draw [-{Stealth[length=3mm, width=2mm]}] (120:4.9) -- (120:5);
\draw [dashed] (120:4.9) -- (0,0);
\draw [-{Stealth[length=3mm, width=2mm]}] (0:4.9) -- (0:5);

\end{tikzpicture}
\caption{Dots denote the edges of $\Gamma\backslash\mathcal{B}$}\label{fig:alphabets}
\end{figure}
\end{center}

%edge position explain..?

\begin{center}
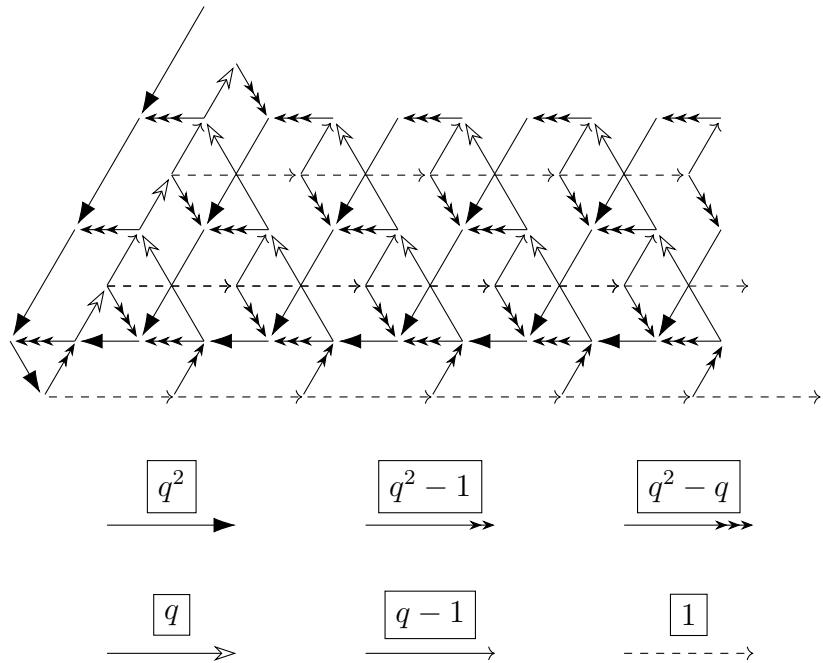
\begin{figure}[H]
\newcommand*\rows{4}
\begin{tikzpicture}[scale=1.7]
  
\foreach \row in {0, 1, ...,\rows} {
\draw [-{Stealth}{Stealth}] (0.52+\row,0.02) --++ (60:0.46);
\draw [-{Stealth}{Stealth}{Stealth}] (0.749+\row,0.25*1.74) --++ (180:0.46);

\draw [->] (1+\row,0.87) --++ (60:0.46);
%\draw[->,thick] (0.25+\row,0.25*1.74) --++ (300:0.46);
\draw [-{Latex[length=3mm, width=2mm]}] (1.249+\row,0.25*1.74) --++ (180:0.46);
\draw [->,dashed] (0.55+\row,0) --++ (0:0.96);
\draw [-{Stealth}{Stealth}{Stealth}] (1.249+\row,0.75*1.74) --++ (180:0.46);
\draw [->] (1.5+\row,1.74) --++ (60:0.46);
\draw [-{Stealth}{Stealth}{Stealth}] (1.749+\row,1.25*1.74) --++ (180:0.46);
\draw [-{Latex[length=3mm, width=2mm]}] (1.25+\row,1.25*1.74) [bend right]--++ (240:0.96);
\draw [-{Latex[length=3mm, width=2mm]}] (0.75+\row,0.75*1.74) --++ (240:0.96);
\draw [-{Stealth[open, length=3mm, width=2mm]}] (1.75+\row,0.25*1.74) --++ (120:0.96);
\draw [->,dashed] (1+\row,0.5*1.73) --++ (0:0.96);

%\draw [] (0.515+\row,0) circle (0.02cm);
    }
    
\foreach \row in {0, 1, 2, 3}{
\draw [-{Stealth[open, length=3mm, width=2mm]}] (2.25+\row,0.75*1.74) --++ (120:0.96);
\draw [->,dashed] (1.5+\row,1.73) --++ (0:0.96);
\draw [->,dashed] (1+\row,0.5*1.73) --++ (0:0.96);
\draw [-{Stealth}{Stealth}{Stealth}] (1+\row,0.5*1.73) --++ (300:0.46);
\draw [-{Stealth}{Stealth}{Stealth}] (1.5+\row,1.73) --++ (300:0.46);
}
\draw [-{Stealth[open, length=3mm, width=2mm]}] (1.25,0.75*1.74) --++ (60:0.46);
\draw [-{Stealth[open, length=3mm, width=2mm]}] (1.75,1.25*1.74) --++ (60:0.46);
\draw [-{Stealth[open, length=3mm, width=2mm]}] (0.75,0.25*1.74) --++ (60:0.46);
\draw [-{Stealth}{Stealth}] (5.53,0) --++ (60:0.46);
\draw [-{Stealth}{Stealth}{Stealth}] (5.75,0.25*1.74) --++ (180:0.47);
\draw [-{Stealth}{Stealth}{Stealth}] (5,0.5*1.74) --++ (300:0.47);
\draw [-{Latex[length=3mm, width=2mm]}] (0.25,0.25*1.74) --++ (300:0.47);
\draw [-{Stealth}{Stealth}{Stealth}] (5.5,1.73) --++ (300:0.47);
\draw [-{Latex[length=3mm, width=2mm]}] (1.75,1.75*1.74) --++ (240:0.97);
\draw [-{Latex[length=3mm, width=2mm]}] (5.75,0.75*1.74) --++ (240:0.97);
\draw [->,dashed] (5.56,0) --++ (0:0.96);

\draw [-{Latex[length=3mm, width=2mm]}] (1,-1) -- (2,-1);
\node [draw] at (1.5,-0.7) {$q^2$};
\draw [-{Stealth}{Stealth}] (3,-1) -- (4,-1);
\node [draw] at (3.5,-0.7) {$q^2-1$};
\draw [-{Stealth}{Stealth}{Stealth}] (5,-1)--(6,-1); 
\draw [-{Stealth}{Stealth}{Stealth}] (2,2.6)--+(300:0.47); 
\node [draw] at (5.5,-0.7) {$q^2-q$};
%\draw [-{Stealth[length=3mm, width=2mm]}] (0,-2) -- (1,-2);
\draw [-{Stealth[open, length=3mm, width=2mm]}] (1,-2) -- (2,-2);
%\draw [-{Latex[open, length=3mm, width=2mm]}] (4,-2) -- (5,-2);
\draw [->] (3,-2) -- (4,-2);
\draw [->,dashed] (5,-2) -- (6,-2);
\node [draw] at (1.5,-1.7) {$q$};
\node [draw] at (3.5,-1.7) {$q-1$};
\node [draw] at (5.5,-1.7) {$1$};

%\draw [-{Square[open, length=3mm, width=2mm]}] (7,0) -- (8,0);
\end{tikzpicture}
\caption{Global description of the weights $w$ of out-neighbors}\label{fig:descripD}
\end{figure}
\end{center}

%The directed graph with weights depicted in Figure~\ref{fig:descripD} provides a description of the Markov shift $(\mathcal{D},\sigma)$. In this graph, the (implicitly presented) vertices correspond to the elements of $\Gamma\backslash G/I$, and a directed edge exists from vertex $e_{k,\ell}$ to vertex $e_{k',\ell'}$ only if the pair $(e_{k,\ell}, e_{k',\ell'})$ satisfies $t(e_{k,\ell})=s(e_{k,\ell'})$. The weights assigned to the edges represent the number of admissible occurrences, that is, a pair $(e_{k,\ell}, e_{k',\ell'})$ lifts to a pair $(e,e')$ of edges in $\mathcal{B}$ such that $s(e),s(e'),t(e')$ do not form a chamber in $\mathcal{B}$. 

We will conclude this section by proving the following lemma, which will be used in demonstrating the convergence of the zeta function.

\begin{lem}\label{lem:finiteness}
For any positive integer $n$, there are finitely many edges $e$ in $X$ that form a geodesic cycle of length $n$ in $X$.
\end{lem}

\begin{proof}
Since the proofs for type $1$-cycles and type $2$-cycles are similar, it suffices to consider the type $1$ case. For a chamber $(v_{m,n},v_{m+1,n},v_{m+1,n+1})$, we label the edges of $X$ in the following manner:
\begin{equation*}
\begin{split}
e_{n,m-n,1}&:=v_{m,n}\rightarrow v_{m+1,n}\\
e_{n,m-n,2}&:=v_{m+1,n}\rightarrow v_{m+1,n+1}\\
e_{n,m-n,3}&:=v_{m+1,n+1}\rightarrow v_{m,n}.
\end{split}
\end{equation*}

For any edge $e_{m',n',i}$ in a cycle $c$ with $n'>n$, the cycle $c$ must contain the sequence of edges $e_{m'',0,1},\cdots e_{m'',n''1}$ with $n''\geq n'$. Otherwise, the path would not return to $e_{m',n',i}$ and thus could not form a cycle. Therefore, the length of cycle $c$ is at least $n$. 

Similarly, for any edge $e_{m',n',i}$ in a cycle $c'$ with $m'>n$, an edge of type $e_{m'',n'',3}$ with $m''>n$ must appear after $e_{m',n'i}$ for $c'$ to form a cycle. This cycle $c$ must contain the sequence of edges $e_{m'',n'',3},e_{m''-1,n'',3},\cdots,e_{0,n'',3}$ making its length at least $n$. Thus we obtain the Lemma.
\end{proof}

%%------
%% Section 4
%%------

\section{Edge zeta function of weighted complexes}\label{sec:4}

In this section, we begin by reviewing the definition of zeta functions for simplicial complexes, paying particular attention to the extension of these functions to the weighted complexes. After we introduce the notion of weighted complexes, we define the edge zeta function for those and express it in terms of a determinant, analogous to the Bass-Ihara formula. We also establish the convergence properties of the zeta function for $\operatorname{PGL}(3,\mathbb{F}_q[t])\backslash\mathcal{B}$. 

\subsection{Zeta function of finite complex revisited}

Let us recall the edge zeta function of finite complexes obtained from a quotient by discrete subgroup $\Gamma$ of $G=\operatorname{PGL}(3,F)$. It is defined in the form of an Euler product. In order to explain the definition, we need to state the \emph{algebraic length} of a cycle. 

In this subsection, we follow the notation from Section~3 to 6 of \cite{KL}. For each element $gZ \in \operatorname{PGL}(3,F)$, there is a scalar $z\in F$ such that $g'=gz$ is a matrix in the complement of $\pi M(3,\mathcal{O})$ in $M(3,\mathcal{O})$. We call $g'$ a minimally integral matrix associated to $g$ and it is unique up to multiplication by $\mathcal{O}^\times$. Define the \emph{algebraic length} of $g$ by $\ell_A(g)=\operatorname{ord}_\pi(\det(g'))$. We always have $\ell_A(g_1g_2)\le \ell_A(g_1)+\ell_A(g_2)$.

A geodesic cycle in $\Gamma\backslash \mathcal{B}$ starting at the vertex $\Gamma g\mathcal{O}^3$ can be lifted to a geodesic in $\mathcal{B}$ starting at $g\mathcal{O}^3$ and ending at $\gamma g\mathcal{O}^3$ for some $\gamma\in \Gamma$. Let us denote by $\kappa_\gamma(g\mathcal{O}^3)$ the homotopy class of the 1-geodesics from $g\mathcal{O}^3$ to $\gamma g\mathcal{O}^3$ in $\mathcal{B}$. The algebraic length of a homotopy class $\kappa_\gamma(g\mathcal{O}^3)$ of $X_\Gamma$ are those of $\kappa_\gamma(g\mathcal{O}^3)$ in $\mathcal{B}$. In other words, if $g^{-1}\gamma g\in K\operatorname{diag}(\pi^{m+n},\pi^n,1)KZ$, then $\kappa_\gamma(g\mathcal{O}^3)$ has algebraic length $\ell_A(\kappa_\gamma(g\mathcal{O}^3))=m+2n$. We note that $k\ell(c)=\ell_A(c)$ when $c$ is an admissible cycle of type $k$. Since $\ell_A(\kappa_\gamma(g\mathcal{O}^3))=\operatorname{ord}_\pi\det\gamma$ (mod $3$), it follows that $\ell_A(\kappa_\gamma(g\mathcal{O}^3))=\ell_A([\gamma])+3m$ for some non-negative integer $m$. See Figure~\ref{fig:alglen} for a comparison between the geometric length and the algebraic length between $g\mathcal{O}^3$ and $g \operatorname{diag}( \pi^{2+3},\pi^3,1)\mathcal{O}^3$.

\begin{center}
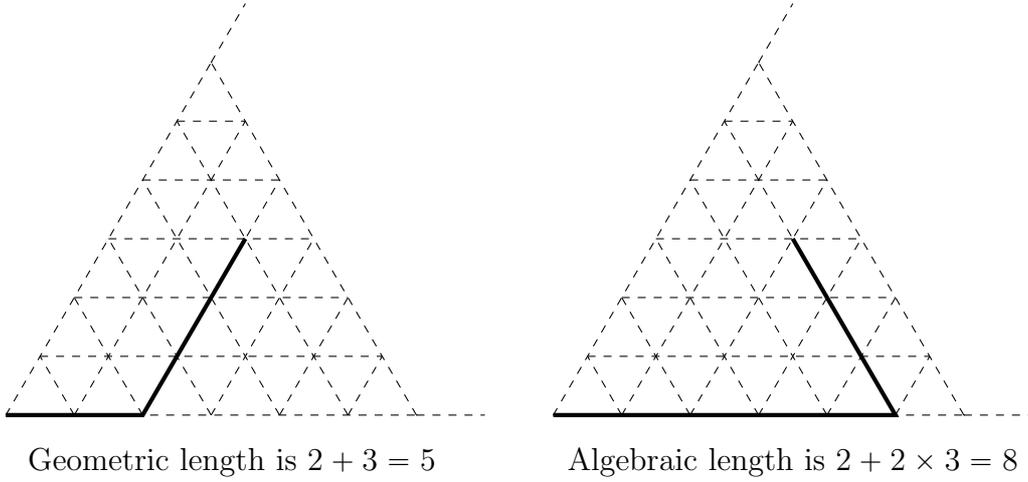
\begin{figure}[h]
\newcommand*\rows{6}
\begin{tikzpicture}[scale=0.9]
    \foreach \row in {0, 1, ...,\rows} {
        \draw[dashed] ($\row*(0.5, {0.5*sqrt(3)})$) -- ($(\rows,0)+\row*(-0.5, {0.5*sqrt(3)})$);
        \draw[dashed] ($\row*(1, 0)$) -- ($(\rows/2,{\rows/2*sqrt(3)})+\row*(0.5,{-0.5*sqrt(3)})$);
        \draw[dashed] ($\row*(1, 0)$) -- ($(0,0)+\row*(0.5,{0.5*sqrt(3)})$);
    }
\draw[dashed] (60:6) --++(60:1);
\draw[dashed] (0:6) --++(0:1);
\draw[ultra thick] (0,0) -- (2,0) --++(60:3);
%\draw[fill] (30:3.466) circle (0.05);
%\node (a) at (30:6) {$\cdots$};
%\node (b) at (3:6.8) {$m$};
%\node (c) at (117:6.8) {$n$};

    \foreach \row in {0, 1, ...,\rows} {
        \draw[dashed] ($\row*(0.5, {0.5*sqrt(3)})+(8,0)$) -- ($(8+\rows,0)+\row*(-0.5, {0.5*sqrt(3)})$);
        \draw[dashed] ($8+\row*(1, 0)$) -- ($(8+\rows/2,{\rows/2*sqrt(3)})+\row*(0.5,{-0.5*sqrt(3)})$);
        \draw[dashed] ($8+\row*(1, 0)$) -- ($(8,0)+\row*(0.5,{0.5*sqrt(3)})$);
    }
\draw[dashed] (11,{3*sqrt(3)}) --++(60:1);
\draw[dashed] (14,0) --++(0:1);
\draw[ultra thick] (8,0) -- (13,0) --++(120:3);
\node at (3.3,-0.7) {Geometric length is $2+3=5$};
\node at (11.5,-0.7) {Algebraic length is $2+2\times 3=8$};

%\draw [-{Stealth[length=3mm, width=2mm]}] (120:6.9) -- (120:7);
%\draw [dashed] (120:6.9) -- (0,0);
%\draw [-{Stealth[length=3mm, width=2mm]}] (0:6.9) -- (0:7);

\end{tikzpicture}
\caption{(Geometric) length and algebraic length}\label{fig:alglen}
\end{figure}
\end{center}

Define the \emph{out-neighbors} of a type $1$ edge $(g_1\mathcal{O}^3,g_2\mathcal{O}^3)$ to be the type 1 edges $(g_2\mathcal{O}^3,g_2\mathcal{O}^3)$ such that $(g_1\mathcal{O}^3,g_2\mathcal{O}^3,g_3\mathcal{O}^3)$ is not a pointed chamber. Each type 1 edge has $q^2$ out-neighbors. Let $E$ be the standard parahoric subgroup $K\cap \sigma K \sigma^{-1}$ of $K$ for
$$\sigma=\begin{pmatrix} & 1 & \\ & & 1 \\ \pi & & \end{pmatrix}.$$ Expressed in terms of $EZ$-cosets, the out-neighbors of a type 1 edge $g EZ$ are given by $g\alpha EZ$, where $\alpha EZ$ are the $EZ$-cosets occurring in the double coset
$$L_E=E\begin{pmatrix} 1 & & \\ & 1 & \\ & & \pi\end{pmatrix} EZ=\bigcup_{x,y\in\mathcal{O}_F/\pi\mathcal{O}_F}\begin{pmatrix} 1 & & \\ & 1 & \\ x\pi & y\pi & \pi \end{pmatrix} EZ.$$ 
The double coset $L_E$ may be naturally regarded as the parahoric operator on $L^2(G/EZ)$ given by
\begin{align*}
L_E\colon L^2(G/EZ)\to&\, L^2(G/EZ)\\
f\mapsto&\,L_Ef(gEZ)=\sum_{x,y\in\mathcal{O}_F/\pi\mathcal{O}_F}f\Bigg(g\begin{pmatrix} 1 & & \\ & 1 & \\ x\pi & y\pi & \pi \end{pmatrix} EZ\Bigg).
\end{align*}

For $k=1,2$, the type $k$ edge zeta function of a finite complex $X_\Gamma$ is defined by
$$ Z_{X_\Gamma,k}(q^{-s})=\prod_{[c]\colon\textrm{prime of type }k}\frac{1}{1-q^{-s\ell_A(c)}}$$
where $[c]$ runs over the equivalence classes of primitie admissible cycles of type $k$ in $X_\Gamma$, and $\ell_A(c)$ is the algebraic length of a cycle in $[c]$ which is independent of the choice of the cycle. Let us also denote by $N_n(X_\Gamma)$ the number of closed geodesic cycles of type 1 in $X_\Gamma$ whose algebraic length is $n$. 

\begin{prop}[Proposition~6.1.1 of \cite{KL}] For $k\in\{1,2\}$, the type $k$ edge zeta function has the following expressions.
$$Z_{X_\Gamma,k}(u)=\exp\left(\sum_{n=1}^{\infty}\frac{N_n(X_\Gamma)}{n}u^{kn}\right)=\frac{1}{\det(I-L_Eu^k)}.$$
\end{prop}

\bigskip

\subsection{Weighted complexes and their zeta functions}

For a directed edge $\widetilde{e}$ of type $1$ in $\mathcal{B}$ and a directed edge $e'$ of type $1$ in $X=\operatorname{PGL}(3,\mathbb{F}_q[t])\backslash\mathcal{B}$, let
$$w(\widetilde{e},e')=\#\{\widetilde{e}'\in \operatorname{E}_1(\mathcal{B})\colon \pi(\widetilde{e}')=e',(\widetilde{e},\widetilde{e}') \textrm{ is a geodesic path in }\mathcal{B}\}.$$ It counts the number of type $1$ edges $\widetilde{e}'$ in $\mathcal{B}$ that does not form a chamber with the other one $\widetilde{e}$ in $\mathcal{B}$ and projects to a fixed one $e'$ in $X$. For an edge $e$ of $X=\Gamma\backslash \mathcal{B}$, we write $w(e,e')=w(\widetilde{e},e')$, where $\widetilde{e}$ is any preimage of $e$ in $\mathcal{B}$. This is again well-defined since $w(e,e')$ does not depend on the specific choice of the preimage $\widetilde{e}$. In particular, $w(e,e')\ne0$ if and only if $(e,e')$ is an admissible path in $\operatorname{PGL}(3,\mathbb{F}_q[t])\backslash\mathcal{B}$. %and it counts the number of $\Gamma$-equivalence classes of pair $(\widetilde{e},\widetilde{e}')$ of edges in $\mathcal{B}$ such that $s(\widetilde{e}),s(\widetilde{e}'),t(\widetilde{e}')$ do not form a chamber in $\mathcal{B}$.

Let $\operatorname{E}_k(\Gamma\backslash \mathcal{B})$ be the set of all directed edges in $\Gamma\backslash \mathcal{B}$ of type $k$. Given an arbitrary set $I$, we consider the formal complex vector space
$$S(I)=\underset{i\in I}{\bigoplus}\mathbb{C}i,$$ consisting of the elements of the form $\sum_{i\in I}a_i$ where $a_i=0$ except for finitely many $i\in I$.
Define the operator $T_1\colon S(\operatorname{E}_1)\to S(\operatorname{E}_1)$ and $T_2\colon S(\operatorname{E}_2)\to S(\operatorname{E}_2)$ by
$$T_1e=\sum_{e\overset{1}{\to} e'}w(e,e')e'\qquad\textrm{ and }\qquad T_2e=\sum_{e\overset{2}{\to}e'}w(e,e')e'$$ where the sum runs over all same type edges $e'$ with $s(e')=t(e)$. %and %$s(e),s(e'),t(e')$ do not form a chamber in $\mathcal{B}$. 
%We will show that for any given $n\in\mathbb{N}$, the operator $T_k^n$ is traceable, and the trace is given by
%$$\operatorname{Tr}(T_k^n)=\sum_{c\colon\ell(c)=n}w(c)\ell(c_0)$$
%where $c$ runs over type $k$-cycles and $c_0$ is the underline primitive cycle of a given cycle $c$.

\begin{lem}\label{lem:4.2}
For any given $n\in\mathbb{N}$, the operator $T_k^n$ is traceable, and the trace is given by
$$\operatorname{Tr}(T_k^n)=\sum_{c\colon\ell(c)=n}w(c)\ell(c_0),$$
where $c$ runs over type $k$-cycles and $c_0$ is the underline primitive cycle of a given cycle $c$.
\end{lem}
\begin{proof}
%To prove lemma, we claim that there are no $k$-cycles containing edges whose starting vertices are outside the parallel with vertices $v_{0,0},v_{n,0},v_{n,n},v_{2n,n}.$
An edge $e\in\operatorname{E}_k(\Gamma\backslash\mathcal{B})$ contributes to the trace of $T_k^n$ only if $\langle T_k^ne,e\rangle \ne 0$. Since we may view $T_k$ as an operator that sends potentials from an edge to the following edges and therefore  $\langle T_k^ne,e\rangle\ne 0$ if and only if $e$ lies on some cycle of length $n$. Lemma~\ref{lem:finiteness} shows that the number of edges contributing to the trace $T_k^n$ is finite.
%From Figure~\ref{fig:descripD},  
From this, it follows that $\sum_e\langle T_k^ne,e\rangle$ is finite and hence $T^n$ is traceable.
\end{proof}

Now for a closed path $p=(e_1,e_2,\ldots,e_n)$, let 
$$w(p)=\prod_{j\textrm{ mod }n}w(e_j,e_{j+1}).$$ 

We define the geometric zeta function of the discrete subgroup $\Gamma$ of $\operatorname{PGL}_3$ by
$$Z_\Gamma(q^{-s})=Z_{\Gamma,1}(q^{-s})Z_{\Gamma,2}(q^{-s})=\prod_{k=1}^2\prod_{c\colon\textrm{prime of type }k}\frac{1}{1-w(c)q^{-s\ell_A(c)}}$$ where $\ell_A(c)$ is the algebraic length of the cycle $c$ and $w(c)$ is the same weight previously defined in the case of graphs. With a slight modification to the case for finite complexes, let
$$N_n(X_\Gamma)=\sum_{c\colon \ell_A(c)=n}w(c)\ell_A(c_0)$$ be the \emph{weigthed} number of closed geodesic cycles of type 1 in $X_\Gamma$ whose algebraic length is $n$.

\begin{prop}For $k\in\{1,2\}$, the type $k$ edge zeta function has the following expressions:
$$Z_{\Gamma,k}(u)=\exp\left(\sum_{n=1}^{\infty}\frac{N_n(X_\Gamma)}{n}u^{kn}\right).$$
\end{prop}
\begin{proof} %By combining the proofs of Proposition~6.1.1 of \cite{KL} and AA of \cite{DK18}, a similar argument yields the proof.

Given a cycle $c$, let us denote by $c_0$ the underlying prime cycle of $c$. Then, we have
\begin{align*}
Z_{\Gamma,k}(q^{-s})^{-1}=&\,\prod_{c_0}\bigr(1-w(c_0)q^{-sk\ell(c_0)}\bigr)=\exp\left(\sum_{c_0}\log(1-w(c_0)q^{-sk\ell(c_0)})\right)\\
=&\,\exp\left(-\sum_{c_0}\sum_{n=1}^{\infty}\frac{w(c_0)^nq^{-skn\ell(c_0)}}{n}\right)=\exp\left(-\sum_{c}\frac{w(c)q^{-sk\ell(c)}}{\ell(c)}\ell(c_0)\right)\\=&\,\exp\left(-\sum_{n=1}^{\infty}\frac{q^{-skn}}{n}\sum_{c\colon \ell(c)=n}w(c)\ell(c_0)\right).
\end{align*}
This proves the result.
\end{proof}

\subsection{Covergence as a determinant}

\begin{lem}
There is $\alpha>0$ such that for $s\in\mathbb{C}$ with $|q^{-s}|<\alpha$, the series 
$$-\sum_{n=1}^{\infty}\frac{q^{-sn}}{n}T_k^n$$ converges weakly. Let $\log(1-uT_k)$ be its weak limit. This is traceable and we have
$$Z_{\Gamma,k}(q^{-s})^{-1}=\exp\bigr(\!\operatorname{Tr}(\log(1-uT_k))\bigr)$$ for every $s\in\mathbb{C}$ with $|q^{-s}|<\alpha$.
\end{lem}
\begin{proof}
We note the $\sum_{e'}w(e,e')=q^2$ for every $e\in\operatorname{E}_k(\Gamma\backslash \mathcal{B})$. It follows that for any two $e_i,e_j\in\operatorname{E}_k(\Gamma\backslash\mathcal{B})$, we have
$$|\langle T_k^ne_i,e_j\rangle|\le q^{2n}.$$ Then, the trace assertion is in fact an assertion of changing the order of summation because the trace is itself a sum over $\operatorname{E}_k(\Gamma\backslash\mathcal{B})$. For $q^{-s}>0$, all summands are positive, thus we may change the order of summation; for general $q^{-s}$, we use absolute convergence to reach the same conclusion.
\end{proof}

Summarizing the discussion up to this point, we obtain the following.

\begin{prop}For $k\in\{1,2\}$, the type $k$ edge zeta function has the following expressions.
$$Z_{\Gamma,k}(u)=\frac{1}{\det(I-T_ku^k)}.$$
\end{prop}

\bigskip

%%-----
%% Section 5
%%-----

\section{Main result}\label{sec:5}

In this section, we present the main result of the paper, an closed formula for the edge zeta function of the standard arithmetic quotient of $\operatorname{PGL}_3$ over a field of formal series. Additionally, we provide an exact formula for the number of closed cycles in this setting. The main strategy is to compute type 1 zeta function by applying trunctation in a specific direction of the matrix and then taking the limit.

%we compute the determinant of $I-T_1u$ and $I-T_2u^2$. Under the choice of labelling oriented edges, the matrix representation of $I-T_2$ coincides with the matrix representation of $I-T_1$. Using this, it is possible to obtain the determinant of $I-T_2u^2$ from the determinant of $I-T_1u$.

\subsection{Truncation: The determinant of $I-uT_{1,k}$}

Let $X_k$ be the subcomplex of $X$ containing vertices $v_{m,n}$ with $n\leq k$, excluding the chambers of the form $(v_{m,k-1},v_{m,k},v_{m+1,k})$. 
Given a chamber $(v_{m,n},v_{m+1,n},v_{m+1,n+1})$, we label the edges of $X_k$ by  (see Figure~\ref{fig:edgelabel_X} and \ref{fig:X_3})
\begin{equation*}
\begin{split}
e_{n,m-n,1}&:=v_{m,n}\rightarrow v_{m+1,n}\\
e_{n,m-n,2}&:=v_{m+1,n}\rightarrow v_{m+1,n+1}\\
e_{n,m-n,3}&:=v_{m+1,n+1}\rightarrow v_{m,n}.
\end{split}
\end{equation*}

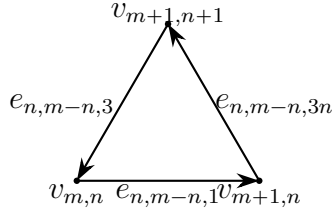
\begin{figure}[H]
\begin{center}
\begin{tikzpicture}[scale=0.4]
% lv 4 pic
\draw[-{Stealth[length=3mm, width=2mm]},thick] (-3,5)--(3,5);\draw[-{Stealth[length=3mm, width=2mm]},thick] (3,5)--(0,+{sqrt(27)}+5);\draw[-{Stealth[length=3mm, width=2mm]},thick] (0,+{sqrt(27)}+5)--(-3,5);
\fill (0,+{sqrt(27)}+5) circle (3pt);\fill (-3,5) circle (3pt); \fill (3,5) circle (3pt);
\node at (-3,4.5) {$v_{m,n}$};\node at (3,4.5) {$v_{m+1,n}$};\node at (0,10.5) {$v_{m+1,n+1}$}; \node at (0,4.5) {$e_{n,m-n,1}$};
\node at(-3.5,7.5) {$e_{n,m-n,3}$};\node at(3.5,7.5) {$e_{n,m-n,3n}$};
\end{tikzpicture}
\end{center}
\caption{Labelling of type 1 the oriented edges}\label{fig:edgelabel_X}
\end{figure}

\begin{figure}[h]
\begin{center}
\begin{tikzpicture}[scale=1]
% lv 4 pic
\draw[-{Stealth[length=3mm, width=2mm]},thick] (-6,0)--(-4,0);\draw[-{Stealth[length=3mm, width=2mm]},thick] (-4,0)--(-2,0);\draw[-{Stealth[length=3mm, width=2mm]},thick] (-2,0)--(0,0);\draw[thick,->](0,0)--(2,0);\draw[-{Stealth[length=3mm, width=2mm]},thick] (2,0)--(4,0); \draw[-{Stealth[length=3mm, width=2mm]},thick] (4,0)--(6,0);
\draw[-{Stealth[length=3mm, width=2mm]},thick]  (-5,{sqrt(3)})--(-3,{sqrt(3)});\draw[-{Stealth[length=3mm, width=2mm]},thick] (-3,{sqrt(3)})--(-1,{sqrt(3)});\draw[-{Stealth[length=3mm, width=2mm]},thick] (-1,{sqrt(3)})--(1,{sqrt(3)}); \draw[-{Stealth[length=3mm, width=2mm]},thick] (1,{sqrt(3)})--(3,{sqrt(3)});\draw[-{Stealth[length=3mm, width=2mm]},thick] (3,{sqrt(3)})--(5,{sqrt(3)});\draw[thick](5,{sqrt(3)})--(6,{sqrt(3)});
\draw[-{Stealth[length=3mm, width=2mm]},thick]  (-4,{sqrt(12)})--(-2,{sqrt(12)});\draw[-{Stealth[length=3mm, width=2mm]},thick] (-2,{sqrt(12)})--(0,{sqrt(12)});\draw[-{Stealth[length=3mm, width=2mm]},thick] (0,{sqrt(12)})--(2,{sqrt(12)});\draw[-{Stealth[length=3mm, width=2mm]},thick] (2,{sqrt(12)})--(4,{sqrt(12)});\draw[-{Stealth[length=3mm, width=2mm]},thick] (4,{sqrt(12)})--(6,{sqrt(12)});
\draw[-{Stealth[length=3mm, width=2mm]},thick] (-5,{sqrt(3)})--(-6,0); \draw[-{Stealth[length=3mm, width=2mm]},thick] (-3,{sqrt(3)})--(-4,0); \draw[-{Stealth[length=3mm, width=2mm]},thick] (-1,{sqrt(3)})--(-2,0); \draw[-{Stealth[length=3mm, width=2mm]},thick] (1,{sqrt(3)})--(0,0); \draw[-{Stealth[length=3mm, width=2mm]},thick] (3,{sqrt(3)})--(2,0);\draw[-{Stealth[length=3mm, width=2mm]},thick] (5,{sqrt(3)})--(4,0);
\draw[-{Stealth[length=3mm, width=2mm]},thick]  (-4,0)--(-5,{sqrt(3)}); \draw[-{Stealth[length=3mm, width=2mm]},thick] (-2,0)--(-3,{sqrt(3)});\draw[-{Stealth[length=3mm, width=2mm]},thick] (0,0)--(-1,{sqrt(3)});\draw[-{Stealth[length=3mm, width=2mm]},thick] (2,0)--(1,{sqrt(3)}); \draw[-{Stealth[length=3mm, width=2mm]},thick] (4,0)--(3,{sqrt(3)}); \draw[-{Stealth[length=3mm, width=2mm]},thick] (6,0)--(5,{sqrt(3)});
\draw[-{Stealth[length=3mm, width=2mm]},thick] (-3,{sqrt(3)})--(-4,{sqrt(12)}); \draw[-{Stealth[length=3mm, width=2mm]},thick] (-1,{sqrt(3)})--(-2,{sqrt(12)}); \draw[-{Stealth[length=3mm, width=2mm]},thick] (1,{sqrt(3)})--(0,{sqrt(12)}); \draw[-{Stealth[length=3mm, width=2mm]},thick] (3,{sqrt(3)})--(2,{sqrt(12)});\draw[-{Stealth[length=3mm, width=2mm]},thick] (5,{sqrt(3)})--(4,{sqrt(12)});
\draw[-{Stealth[length=3mm, width=2mm]},thick]  (-4,{sqrt(12)})--(-5,{sqrt(3)}); \draw[-{Stealth[length=3mm, width=2mm]},thick] (-2,{sqrt(12)})--(-3,{sqrt(3)});\draw[-{Stealth[length=3mm, width=2mm]},thick] (0,{sqrt(12)})--(-1,{sqrt(3)});\draw[-{Stealth[length=3mm, width=2mm]},thick] (2,{sqrt(12)})--(1,{sqrt(3)}); \draw[-{Stealth[length=3mm, width=2mm]},thick] (4,{sqrt(12)})--(3,{sqrt(3)}); \draw[-{Stealth[length=3mm, width=2mm]},thick] (6,{sqrt(12)})--(5,{sqrt(3)});
\draw[-{Stealth[length=3mm, width=2mm]},thick] (-3,{sqrt(27)})--(-4,{sqrt(12)}); \draw[-{Stealth[length=3mm, width=2mm]},thick] (-1,{sqrt(27)})--(-2,{sqrt(12)}); \draw[-{Stealth[length=3mm, width=2mm]},thick] (1,{sqrt(27)})--(0,{sqrt(12)}); \draw[-{Stealth[length=3mm, width=2mm]},thick] (3,{sqrt(27)})--(2,{sqrt(12)});\draw[-{Stealth[length=3mm, width=2mm]},thick] (5,{sqrt(27)})--(4,{sqrt(12)});
 \draw[-{Stealth[length=3mm, width=2mm]},thick] (-2,{sqrt(12)})--(-3,{sqrt(27)});\draw[-{Stealth[length=3mm, width=2mm]},thick] (0,{sqrt(12)})--(-1,{sqrt(27)});\draw[-{Stealth[length=3mm, width=2mm]},thick] (2,{sqrt(12)})--(1,{sqrt(27)}); \draw[-{Stealth[length=3mm, width=2mm]},thick] (4,{sqrt(12)})--(3,{sqrt(27)}); \draw[-{Stealth[length=3mm, width=2mm]},thick] (6,{sqrt(12)})--(5,{sqrt(27)});
\fill (-3,{sqrt(27)}) circle (2pt);\fill (-1,{sqrt(27)}) circle (2pt);\fill (1,{sqrt(27)}) circle (2pt);\fill (3,{sqrt(27)}) circle (2pt);\fill (5,{sqrt(27)}) circle (2pt);
(3pt)
 \fill (-4,0) circle (3pt);\fill (-6,0) circle (3pt);\ \fill (-4,0) circle (3pt);\fill (-2,0) circle (3pt); \fill (0,0) circle (3pt); \fill (2,0) circle (3pt); \fill (4,0) circle (3pt); \fill (6,0) circle (3pt);
\fill (-5,{sqrt(3)}) circle (3pt);\fill (-3,{sqrt(3)}) circle (3pt);\fill (-1,{sqrt(3)}) circle (3pt);\fill (1,{sqrt(3)}) circle (3pt);\fill (3,{sqrt(3)}) circle (3pt);\fill (5,{sqrt(3)}) circle (3pt);
\fill (-4,+{sqrt(12)}) circle (3pt);\fill (-2,+{sqrt(12)}) circle (3pt); \fill (0,+{sqrt(12)}) circle (3pt); \fill (2,+{sqrt(12)}) circle (3pt); \fill (4,+{sqrt(12)}) circle (3pt); \fill (6,+{sqrt(12)}) circle (3pt);
\fill (-3,{sqrt(27)}) circle (3pt);\fill (-1,{sqrt(27)}) circle (3pt);\fill (1,{sqrt(27)}) circle (3pt);\fill (3,{sqrt(27)}) circle (3pt);\fill (5,{sqrt(27)}) circle (3pt);
\node at (-5,-0.3) {$e_{0,0,1}$};\node at (-3,-0.3) {$e_{0,1,1}$};\node at (-1,-0.3) {$e_{0,2,1}$};\node at (1,-0.3) {$e_{0,3,1}$};\node at (3,-0.3) {$e_{0,4,1}$};\node at (5,-0.3) {$e_{0,5,1}$};\node at (-4,1.5) {$e_{1,0,1}$};\node at (-2,1.5) {$e_{1,1,1}$};\node at (0,1.5) {$e_{1,2,1}$};\node at (2,1.5) {$e_{1,3,1}$};\node at (4,1.5) {$e_{1,4,1}$};\node at (6,1.5) {$e_{1,5,1}$};\node at (-3,3.2) {$e_{2,0,1}$};\node at (-1,3.2) {$e_{2,1,1}$};\node at (1,3.2) {$e_{2,2,1}$};\node at (3,3.2) {$e_{2,3,1}$};\node at (5,3.2) {$e_{2,4,1}$};\node at (-5.5,0.7) {$e_{0,0,3}$};\node at (-4.5,0.7) {$e_{0,0,2}$};\node at (-3.5,0.7) {$e_{0,1,3}$};\node at (-2.5,0.7) {$e_{0,1,2}$};\node at (-1.5,0.7) {$e_{0,2,3}$};\node at (-0.5,0.7) {$e_{0,2,2}$};\node at (0.5,0.7) {$e_{0,3,3}$};\node at (1.5,0.7) {$e_{0,3,2}$};\node at (2.5,0.7) {$e_{0,4,3}$};\node at (3.5,0.7) {$e_{0,4,2}$};\node at (4.5,0.7) {$e_{0,5,3}$};\node at (5.5,0.7) {$e_{0,5,2}$};\node at (-4.5,2.4) {$e_{1,0,3}$};\node at (-3.5,2.4) {$e_{1,0,2}$};\node at (-2.5,2.4) {$e_{1,1,3}$};\node at (-1.5,2.4) {$e_{1,1,2}$};\node at (-0.5,2.4) {$e_{1,2,3}$};
\node at (0.5,2.4) {$e_{1,2,2}$};\node at (1.5,2.4) {$e_{1,3,3}$};\node at (2.5,2.4) {$e_{1,3,2}$};\node at (3.5,2.4) {$e_{1,4,3}$};\node at (4.5,2.4) {$e_{1,4,2}$};
\node at (5.5,2.4) {$e_{1,5,3}$};
\node at (-3.5,4.1) {$e_{2,0,3}$};\node at (-2.5,4.1) {$e_{2,0,2}$};\node at (-1.5,4.1) {$e_{2,1,3}$};\node at (-0.5,4.1) {$e_{2,1,2}$};\node at (0.5,4.1) {$e_{2,2,3}$};
\node at (1.5,4.1) {$e_{2,2,2}$};\node at (2.5,4.1) {$e_{2,3,3}$};\node at (3.5,4.1) {$e_{2,3,2}$};\node at (4.5,4.1) {$e_{2,4,3}$};\node at (5.5,4.1) {$e_{2,4,2}$};
\draw[-,dashed] (-7,{sqrt(27)}) -- (7,{sqrt(27)});
\end{tikzpicture}
\end{center}
\caption{$\operatorname{E}_1(X_3)$}\label{fig:X_3}
\end{figure}

\bigskip
 The operator $T_{1,k}$ on the set of oriented edges of color 1 in $X_k$ is defined by 
\begin{equation*}
\begin{split}
&T_{1,k}e_{m,n,1}:=\begin{cases}
(q^2-1)e_{0,n,2}+e_{0,n+1,1}& \text{ if } m=0\\
(q-1)e_{m,n,2}\\+(q^2-q)e_{m-1,n+1,3}+e_{m,n+1,1}&\text{ if } n\neq 0\\
\end{cases}\\
&T_{1,k}e_{m,n,2}:=\begin{cases}
(q^2-q)e_{m,0,3}+qe_{m+1,0,1}& \text{ if } n=0\\
(q^2-q)e_{m,n,3}+qe_{m+1,n-1,2}&\text{ if }n\neq 0,k-1\\
(q^2-q)e_{k-1,n,3} &\text{ if } n=k-1\\
\end{cases}\\
&T_{1,k}e_{m,n,3}:=\begin{cases}
q^2e_{0,0,1}& \text{ if } m=n=0\\
q^2e_{0,n-1,2}&\text{ if }m=0,n\neq 0\\
q^2e_{m-1,n,3}& \text{otherwise}.\\
\end{cases}\\
\end{split}
\end{equation*}
%Fix a complex number $u$. For any positive integer $s\leq k$, the sequences  $\{a_{k,(s,1),\ell}(u)\}_{\ell}$ satisfy
%\begin{equation*}
%\begin{split}
%&a_{k,(s,1),1}(u):=\begin{cases}-(q^2-1)u&\text{ if } s=1\\
%0 &\text{ otherwise}
%\end{cases}\\
%&a_{k,(s,1),\ell+1}(u):=\begin{cases}-(q^2-1)u+\displaystyle\sum_{i=1}^{k}(q-1)q^{2i+1}u^{i+2}a_{k,(i,1),\ell}(u)&\text{ if } s=1\\
 %                            qu^2a_{k,(s-1,1),\ell}(u)&\text{ otherwise}\end{cases}
%\end{split}\end{equation*}
%
Let $M_{k,N}$ be a $N\times N$ block matrix of the from
\begin{equation}\label{eq:1}
(M_{k,N}(u))_{m,n}:=\begin{cases}
A_{k}(u)&\text{if } m=n=1\\
B_{k}(u)&\text{if } m=n>1\\
C_k(u)&\text{if } m-n=1\\
D_k(u)&\text{if }n-m=1\\
0&\text{otherwise}
\end{cases}
\end{equation}

%We define the $3k\times 3k$ matrices 
where $A_{k}(u), B_{k}(u), C_{k}(u),D_k(u)$ and $E_k(u)$ are matrices defined by%to be
\begin{equation*}
\begin{split}
&(A_{k}(u))_{m,n}:=\begin{cases}
%a_{k,(s,t),\ell}(u)&\text{ if } (m,n)=(3s-1,3t-2)\\
-q^2u&\text{ if } (m,n)=(1,3)\\
-qu&\text{ if }(m,n)=(3s+1,3s-1)\\
(B_k(u))_{m,n}&\text{ otherwise}\end{cases},\\
\end{split}
\end{equation*}
\begin{equation*}
\begin{split}
&(B_{k}(u))_{m,n}:=\begin{cases}1 &\text{ if } m=n\\
-(q^2-1)u&\text{ if } (m,n)=(2,1)\\
-(q-1)u&\text{ if } (m,n)=(3s+2,3s+1)\\
-q^2u&\text{ if } (m,n)=(3(s-1),3s)\\
-(q^2-q)u&\text{ if } (m,n)=(3s,3s-1)\\
0&\text{ otherwise}\end{cases},\\
%&(C_{k,\ell}(u))_{m,n}:=\begin{cases}-u &\text{ if } (m,n)=(3s-2,3s-2)\\
%ua_{k,(s,t),\ell}(u)&\text{ if } (m,n)=(3s-1,3t-2)\\
%\displaystyle \sum_{i=s}^k (q-1)q^{2(i-s)+1}u^{i-s+2}a_{k,(i,t),\ell}(u)&\text{ if } (m,n)=(3s,3t-2)\text{ and }s\geq t\\
%-(q-1)q^{2(t-s)-1}u^{t-s}\\
%+\displaystyle \sum_{i=s}^k (q-1)q^{2(i-s)+1}u^{i-s+2}a_{k,(i,t),\ell}(u)&\text{ if } (m,n)=(3s,3t-2)\text{ and }s< t\\
%0&\text{ otherwise}\end{cases},\\
&(C_k(u))_{m,n}:=\begin{cases}-u &\text{ if } (m,n)=(3s-2,3s-2)\\
-(q^2-q)u&\text{ if } (m,n)=(3s,3s+1)\\
0&\text{ otherwise}\end{cases}\\
%\end{split}
%\end{equation*}
%and 
%\begin{equation*}
&(D_k(u))_{m,n}:=\begin{cases}-qu &\text{ if } (m,n)=(3s+2,3s-1)\\
-q^2u&\text{ if } (m,n)=(2,3)\\
0&\text{ otherwise}\end{cases}.
\end{split}
\end{equation*}
%Let $\alpha_{k,(s,t)}(u):=\lim_{\ell\rightarrow \infty} a_{k,(s,t),\ell}(u)$ and
%\begin{equation*}
%(A_{k}(u))_{m,n}:=\begin{cases}1 &\text{ if } m=n\\
%\alpha_{k,(s,t)}(u)&\text{ if } (m,n)=(3s-1,3t-2)\\
%-q^2u&\text{ if } (m,n)=(1,3)\text{ or }(3(s-1),3s)\\
%-(q^2-q)u&\text{ if } (m,n)=(3s,3s-1)\\
%-qu&\text{ if }(m,n)=(3s+1,3s-1)\\
%0&\text{ otherwise}\end{cases}.
%\end{equation*}

 The matrix $M_{k,N}$ is the matrix representation of $I-uT_{1,k}$ for edges $e_{m,n,i}$ with $n\leq N$. Precisely, $(3m+3kn+i)$-th column of $M_{k,N}$ is the vector representation of $T_1e_{m,n,i}$. 
 
The determinant of $I-uT_{1,k}$ is the limit 
$$\lim_{N\rightarrow \infty} \det M_{k,N}(u).$$
%For any $\ell\geq 1$, the matrix $(E_{k,\ell}(u)\,\,I)$ is obtained by performing the following elementary row operations on $(C_k(u)\,\, B_{k,\ell})$ in order:
%\begin{itemize}
%\item for any $s$, $(3s-1)$-th row$\rightarrow$$(3s-1)$-th row -$\displaystyle\sum_{i=1}^{k}\left\{a_{k,(s,i)}\times(3i-2)\text{-th row}\right\}$,
%\item for any $s$, $3s$-th row$\rightarrow$$3s$-th row $+\{(q^2-q)u\times (3s-1)$-th row$\}$,
%\item and for any $i\geq 1$, $$3(k-i)\text{-row}\rightarrow3(k-i)\text{-th row}+\{q^2u\times3(k-i+1)\text{-th row}\}.$$
%\end{itemize}
Let $B_{k,1}(u)=B_k(u)$ and let $\{B_{k,\ell}(u)\}$ be a sequence of matrices defined by the following recurrence relation 
$$B_{k,\ell}(u)=B_k(u)-D_k(u)B_{k,\ell-1}(u)^{-1}C_k(u).$$
Then we have
\begin{equation*}
\begin{pmatrix}
I&-D_k(u)B_{k,\ell}(u)^{-1}\\
0&B_{k,\ell}(u)^{-1}
\end{pmatrix}
\begin{pmatrix}
B_{k}(u)&D_k(u)\\
C_{k}(u)&B_{k,\ell}(u)
\end{pmatrix}
=\begin{pmatrix}
B_{k,\ell+1}(u)&0\\
B_{k,\ell}(u)^{-1}C_k(u)&I
\end{pmatrix}
\end{equation*}
% in order:
%\begin{itemize}
%\item 2nd row$\rightarrow$ 2nd row+$\{q^2u\times (3k+3)\text{-th row}\}$
%\item and for any $s\leq k-1$, $$(3s+2)\text{-th row}\rightarrow(3s+2)\text{-th row}+\{qu\times(3(k+s)-1)\text{-th row}\}.$$
%\end{itemize}
The matirx $B_{k,\ell}(u)$ is of the form
\begin{equation*}
\left(B_{k,\ell}(u)\right)_{m,n}=\begin{cases}a_{k,(s,t),\ell}(u)&\text{ if }(m,n)=(3s-1,3t-2)\\
(B_{k}(u))_{m,n}&\text{ otherwise},
\end{cases}
\end{equation*}
where the sequences  $\{a_{k,(s,t),\ell}(u)\}_{\ell}$ are defined by
\begin{equation*}
\begin{split}
&a_{k,(s,t),1}(u):=\begin{cases}-(q^2-1)u&\text{ if } s=t=1\\
-(q-1)u&\text{ if } s=t>1\\
                                               0&\text{ otherwise}
\end{cases}\\
%&a_{k,(1,t),\ell+1}(u):=-(q-1)q^{2t-1}u^{t}+\sum_{i=1}^{k}(q-1)q^{2i+1}u^{i+2}a_{k,(i,t),\ell}(u)\\
&a_{k,(s,t),\ell+1}(u):=\begin{cases}
-(q^2-1)u+\displaystyle\sum_{i=1}^{k}(q-1)q^{2i+1}u^{i+2}a_{k,(i,1),\ell}(u)&\text{ if } s=t=1\\
-(q-1)q^{2t-1}u^{t}+\displaystyle\sum_{i=1}^{k}(q-1)q^{2i+1}u^{i+2}a_{k,(i,t),\ell}(u)&\text{ if }s=1,t\neq 1\\
-(q-1)u+qu^2a_{k,(s-1,t),\ell}(u)& \text{ if } s=t\geq 2\\
qu^2a_{k,(s-1,t),\ell}(u)&\text{ otherwise}.
\end{cases} 
\end{split}
\end{equation*}
 Using elementary row operation, we have $\det B_{k,\ell}(u)=\det B_{k,\ell}(u)^{-1}=1$. 
This shows that $\det M_{k,N}(u)=\det A_{k,N}(u)$ and $\det A_k'(u)=\det I-uT_{1,k}$, where 
$$A_{k,N}(u):=A_k(u)-D_k(u)B_{k,N-1}(u)^{-1}C_k(u)\text{ and }A_k'(u):=\lim_{N\rightarrow \infty}A_{k,N}(u).$$
The matrix $A_{k,N}$ is of the form
\begin{equation*}
\left(A_{k,N}(u)\right)_{m,n}=\begin{cases}a_{k,(s,t),N}(u)&\text{ if }(m,n)=(3s-1,3t-2)\\
(A_{k}(u))_{m,n}&\text{ otherwise},
\end{cases}
\end{equation*}
Let $\alpha_{k,(s,t)}(u):=\lim_{\ell\rightarrow\infty} a_{k,(s,t),\ell}(u).$
% Since $$A_k'=A_1-D_k(u)B_k'(u)^{-1}C_k(u),$$
 For any $s\geq 2,$
\begin{equation}\label{eq:5.2}
\alpha_{k,(s,t)}(u)=\begin{cases}qu^2\alpha_{k,(s-1,t)}(u)&\text{ if } s\neq t\text{ and } s\geq 2\\
-(q-1)u+qu^2\alpha_{k,(s-1,t)}(u)& \text{ if } s=t\geq 2.\end{cases}
\end{equation}
If follows from \eqref{eq:5.2} that $\overline{A}_k(u)=E_{k,1}(u)E_{k,2}(u)A'_k(u)$, where
\begin{equation*}
\begin{split}
(E_{k,1}(u))_{m,n}&:=\begin{cases}1&\text{if }m=n\\-u&\text{if }(m,n)=(3s+2,3s+1)\\0&\text{otherwise}\end{cases}\\
(E_{k,2}(u))_{m,n}&:=\begin{cases}1&\text{if }m=n\\-qu^2&\text{if }(m,n)=(3s+2,3s-1)\\0&\text{otherwise}\end{cases}\\
(\overline{A}_k(u))_{m,n}:&=\begin{cases}-qu&\text{ if }(m,n)=(3s+2,3s+1)\\
\alpha_{k,(1,t)}(u)&\text{if }(m,n)=(2,3t-2)\\
(A_k(u))_{m,n}&\text{ otherwise}.
\end{cases}
\end{split}
\end{equation*}
We have $\det A'_k(u)=\det \overline{A}_k(u).$ Let $[\overline{A}_k(u)]_\ell$ be the $\ell$-th column of $\overline{A}_k(u)$ and $\overline{B}_1(u)=[\overline{A}_k(u)]_{3}$. 
Let $\overline{A}_k'(u)$ be the matrix obtained by sequentially applying the following column operation: for any positive integer $\ell\leq k-1$
\begin{equation*}
\begin{split} 
[\overline{A}_k(u)]_{3(\ell+1)}\longrightarrow \overline{B}_{\ell+1}(u):=[\overline{A}_k(u)]_{3(\ell+1)}+qu^2\overline{B}_{\ell}(u).
\end{split}
\end{equation*}
%where 
%\begin{equation*}&[\overline{A}_k(u)]_{3(k-\ell+1)-1}\longrightarrow[\overline{A}_k(u)]_{3(k-\ell+1)-1}+(q^2-q)u[\overline{A}_k(u)]_{3(k-\ell+1)}\\

%(A''_k(u))_{mn}:=\begin{cases}-qu&\text{ if }(m,n)=(3s+2,3s+1),s\geq 1\\
%\alpha_{k,(1,t)}(u)&\text{if }(m,n)=(2,3t-2)\\
%(A_k(u))_{m,n}&\text{ otherwise}.
%\end{cases}
%\end{equation*}
Let $\widetilde{D}_1(u)=[\overline{A}_k(u)]_{3k-1}$. Let $\widetilde{A}_k(u)$ be the matrix obtained by sequentially applying the following column operation: for any positive integer $\ell\leq k-1$
\begin{equation*}\label{eq:4.2}
\begin{split}
&\widetilde{D}_\ell(u)\longrightarrow\widetilde{B}_\ell(u):=\widetilde{D}_\ell(u)+(q^2-q)u[\overline{A}_k'(u)]_{3(k-\ell+1)}\\
&[\overline{A}'_k(u)]_{3(k-\ell)+1}\longrightarrow\widetilde{C}_\ell(u):=[\overline{A}_k(u)]_{3(k-\ell)+1}+qu\widetilde{B}_\ell(u)\\
&[\overline{A}'_k(u)]_{3(k-\ell)-1}\longrightarrow\widetilde{D}_{\ell+1}(u):=[\overline{A}_k(u)]_{3(k-\ell)-1}+qu\widetilde{C}_\ell(u).
\end{split}
\end{equation*}
Then, the matrix $\widetilde{A}_k(u)$ is a $2\times 2$ block upper triangular matrix where the first diagonal entry is $2\times 2$ matrix given by the equation below, and the second diagonal entry is the identity matrix. Eventually, we obtain
\begin{equation}\nonumber
\begin{split}
\det A'_k(u)&=\det\begin{pmatrix}1&-\sum_{j=1}^kq^{4j-1}(q-1)u^{3j-1}\\
\alpha_{k,(1,1)}(u)&1+\sum_{j=2}^k q^{2j-3}u^{2j-3}\alpha_{k,(1,j)}(u)\end{pmatrix}\\
&=1+\sum_{j=2}^k q^{2j-3}u^{2j-3}\alpha_{k,(1,j)}(u)+\alpha_{k,(1,1)}(u)\sum_{j=1}^kq^{4j-1}(q-1)u^{3j-1}.
\end{split}
\end{equation}
\subsection{Main Result: The determinant of $I-uT_1$}
For any $s\geq 2$, the limit $\alpha_{k,(s,1)}$ satisfies that
$$\alpha_{k,(s,1)}(u)=qu^2\alpha_{k,(s-1,1)}(u)=q^2u^4\alpha_{k,(s-2,1)}(u)=\cdots =q^{s-1}u^{2s-2}\alpha_{k,(1,1)}(u).$$
Thus we have
\begin{equation*}
\begin{split}
\alpha_{k,(1,1)}(u)&=-(q^2-1)u+\displaystyle\sum_{i=1}^{k}(q-1)q^{2i+1}u^{i+2}\alpha_{k,(i,1)}(u)\\
&=-(q^2-1)u+\displaystyle\sum_{i=1}^{k}(q-1)q^{3i}u^{3i}\alpha_{k,(1,1)}(u)
\end{split}
\end{equation*}
and 
$$\alpha_{k,(1,1)}(u)=\frac{-(q^2-1)u}{1-\displaystyle\sum_{i=1}^{k}(q-1)q^{3i}u^{3i}}.$$
For any $s\geq 2$ and $t> s$,
$$\alpha_{k,(s,t)}(u)=qu^2\alpha_{k,(s-1,t)}(u)=q^2u^4\alpha_{k,(s-2,t)}(u)=\cdots =q^{s-1}u^{2s-2}\alpha_{k,(1,t)}(u)$$
whereas for any $s\geq 2$ and $s\geq t,$
\begin{equation*}
\begin{split}
\alpha_{k,(s,t)}(u)&=qu^2\alpha_{k,(s-1,t)}(u)=\cdots =q^{s-t}u^{2s-2t}\alpha_{k,(t,t)}(u)\\
&=-(q-1)q^{s-t}u^{2s-2t+1}+q^{s-t+1}u^{2s-2t+2}\alpha_{k,(t-1,t)}(u)\\
&=-(q-1)q^{s-t}u^{2s-2t+1}+q^{s-1}u^{2s-2}\alpha_{k,(1,t)}(u).
\end{split}
\end{equation*}
This shows that 
\begin{equation*}
\begin{split}
\alpha_{k,(1,t)}(u)=&-(q-1)q^{2t-1}u^{t}+\sum_{i=1}^{k}(q-1)q^{2i+1}u^{i+2}\alpha_{k,(i,t)}(u)\\
=&-(q-1)q^{2t-1}u^{t}+\sum_{i=1}^{k}(q-1)q^{3i}u^{3i}\alpha_{k,(1,t)}(u)\\
&-\sum_{i=t}^{k}(q-1)^2q^{3i-t+1}u^{3i-2t+3}
\end{split}
\end{equation*}
and 
\begin{equation*}
\alpha_{k,(1,t)}(u)=\frac{-(q-1)q^{2t-1}u^{t}-\sum_{i=t}^{k}(q-1)^2q^{3i-t+1}u^{3i-2t+3}}{1-\sum_{i=1}^{k}(q-1)q^{3i}u^{3i}}.
\end{equation*}
Let $\alpha_t(u)=\lim_{k\rightarrow \infty} \alpha_{k,(1,t)}(u)$. Then we have
\begin{equation*}
\alpha_{1}(u)=\frac{-(q^2-1)u}{1-\frac{(q-1)q^3u^3}{1-q^3u^3}}=\frac{-(q^2-1)(1-q^3u^3)u}{1-q^4u^3}
\end{equation*}
and for any $t\geq 2,$
\begin{equation*}
\begin{split}
\alpha_{t}(u)&=\frac{-(q-1)q^{2t-1}u^{t}-\frac{(q-1)^2q^{2t+1}u^{t+3}}{1-q^3u^3}}{1-\frac{(q-1)q^3u^3}{1-q^3u^3}}\\
&=\frac{-(q-1)q^{2t-1}u^t+(q-1)q^{2t+2}u^{t+3}-(q-1)^2q^{2t+1}u^{t+3}}{1-q^4u^3}\\
&=\frac{-(q-1)q^{2t-1}u^t+(q-1)q^{2t+1}u^{t+3}}{1-q^4u^3}.
\end{split}
\end{equation*}
Therefore,
\begin{equation}\label{eq:2}
\begin{split}
\det(I-uT_1)=&\lim_{k\rightarrow \infty} \det A_k(u)\\
=&1+\sum_{j=2}^\infty q^{2j-3}u^{2j-3}\alpha_{j}(u)+\alpha_1(u)\sum_{j=1}^\infty q^{4j-1}(q-1)u^{3j-1}\\
=&1+\sum_{j=2}^\infty \frac{-(q-1)q^{4j-4}u^{3j-3}+(q-1)q^{4j-2}u^{3j}}{1-q^4u^3}\\
&-\frac{(q^2-1)(1-q^3u^3)u}{1-q^4u^3}\sum_{j=1}^\infty q^{4j-1}(q-1)u^{3j-1}\\
=&1-\frac{(q-1)q^4u^3-(q-1)q^6u^6}{(1-q^4u^3)^2}-\frac{(q^2-1)(1-q^3u^3)(q-1)q^3u^3}{(1-q^4u^3)^2}\\
=&\frac{(1-q^3u^3)(1-q^6u^3)}{(1-q^4u^3)^2}
\end{split}
\end{equation}
which yields the following.
\begin{thm} For $\Gamma=\operatorname{PGL}(3,\mathbb{F}_q[t])$, the type 1 edge zeta function $Z_{\Gamma,1}(q^{-s})$ converges for $s$ with $\operatorname{Re}s>2$, and it is given by 
$$Z_{\Gamma,1}(q^{-s})=\frac{(1-q^{4-3s})^2}{(1-q^{3-3s})(1-q^{6-3s})}.$$
\end{thm}

Let $u=q^{-s}$. If we write 
$$u\frac{Z'_{\Gamma,1}(u)}{Z_{\Gamma,1}(u)}=\sum_{m=1}^{\infty} N_m(\Gamma\backslash \mathcal{B})u^m,$$ then from the Euler product we get 
$$N_m(\Gamma\backslash \mathcal{B})=\sum_{c\colon\ell(c)=m}w(c)\ell(c_0)$$

\begin{coro}[Counting closed admissible cycles]
Let $\Gamma$ be the discrete group $\operatorname{PGL}(3,\mathbb{F}_q[t])$ acting on $\mathcal{B}$ and $N_m(\Gamma\backslash\mathcal{B})$ be defined as above. Then, we have
$$N_{m}(\Gamma\backslash\mathcal{B})=\left\{\begin{array}{ll}3q^{2r}-6q^{\frac{4}{3}r}+3q^{r} & \textrm{ if }m=3r \\ 0 & \textrm{ otherwise}\end{array}\right..$$
\end{coro}
\begin{proof}
By comparing the coefficients of $u^m$ in the following
\begin{align*}
\sum_{m=1}^{\infty}N_m(\Gamma\backslash\mathcal{B})u^m=u\frac{d}{du}\log Z_{\Gamma,1}(u)=u\left(\frac{3q^3u^2}{1-q^3u^3}-\frac{6q^4u^2}{1-q^4u^3}+\frac{3q^6u^2}{1-q^6u^3}\right),
\end{align*}
we obtain the result.
\end{proof}

\subsection{Symmetry: The determinant of $I-uT_{2}$}

Let $Y_k$ be the subcomplex of groups consisting of vertices $v_{m,n}$ with $m-n\leq k$, excluding the chambers of the form $$(v_{m,m-k},v_{m,m-k+1},v_{m+1,n-k+1}).$$ 
Given a chamber $(v_{m,n},v_{m+1,n},v_{m+1,n+1})$ in $Y_k$,, the edges in $Y_k$ are labeled as: (see Figure~\ref{fig:edgelabel_Y} and \ref{fig:Y_3})
\begin{equation*}
\begin{split}
f_{n,m-n,1}&:=v_{m,n}\rightarrow v_{m+1,n+1} \\
f_{n,m-n,2}&:=v_{m+1,n+1}\rightarrow v_{m+1,n} \\
f_{n,m-n,3}&:= v_{m+1,n}\rightarrow v_{m,n}.
\end{split}
\end{equation*}
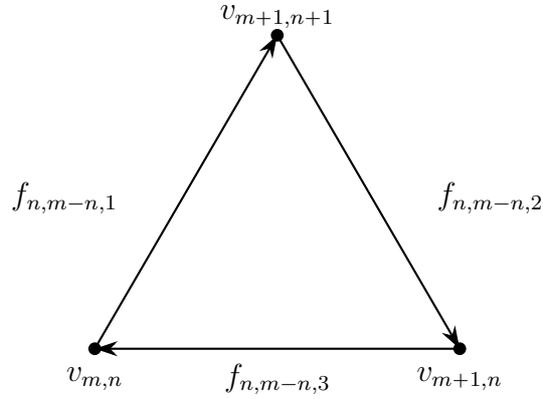
\begin{figure}[H]
\begin{center}
\begin{tikzpicture}[scale=0.8]
% lv 4 pic
\draw [-{Stealth[length=3mm, width=2mm]},thick]  (3,5)--(-3,5);\draw[-{Stealth[length=3mm, width=2mm]},thick] (0,+{sqrt(27)}+5)-- (3,5);\draw [-{Stealth[length=3mm, width=2mm]},thick]  (-3,5)--(0,+{sqrt(27)}+5);
\fill (0,+{sqrt(27)}+5) circle (3pt);\fill (-3,5) circle (3pt); \fill (3,5) circle (3pt);
\node at (-3,4.5) {$v_{m,n}$};\node at (3,4.5) {$v_{m+1,n}$};\node at (0,10.5) {$v_{m+1,n+1}$}; \node at (0,4.5) {$f_{n,m-n,3}$};
\node at(-3.5,7.5) {$f_{n,m-n,1}$};\node at(3.5,7.5) {$f_{n,m-n,2}$};
\end{tikzpicture}
\end{center}
\caption{Labelling of type 2 edges}\label{fig:edgelabel_Y}
\end{figure}

\begin{figure}[h]
\begin{center}
\begin{tikzpicture}[scale=1]
% lv 4 pic
\draw[-{Stealth[length=3mm, width=2mm]},thick]  (-4,0)--(-6,0);\draw[-{Stealth[length=3mm, width=2mm]},thick]  (-2,0)--(-4,0);\draw[-{Stealth[length=3mm, width=2mm]},thick]  (0,0)--(-2,0);
\draw[-{Stealth[length=3mm, width=2mm]},thick]   (-3,{sqrt(3)})--(-5,{sqrt(3)});\draw[-{Stealth[length=3mm, width=2mm]},thick]  (-1,{sqrt(3)})--(-3,{sqrt(3)});\draw[-{Stealth[length=3mm, width=2mm]},thick]  (1,{sqrt(3)})--(-1,{sqrt(3)}); 
\draw[-{Stealth[length=3mm, width=2mm]},thick]  (-2,{sqrt(12)})--(-4,{sqrt(12)});\draw [-{Stealth[length=3mm, width=2mm]},thick] (0,{sqrt(12)})--(-2,{sqrt(12)});\draw[-{Stealth[length=3mm, width=2mm]},thick]  (2,{sqrt(12)})--(0,{sqrt(12)});
\draw[-{Stealth[length=3mm, width=2mm]},thick] (-1,{sqrt(27)})--(-3,{sqrt(27)});\draw[-{Stealth[length=3mm, width=2mm]},thick] (1,{sqrt(27)})--(-1,{sqrt(27)});\draw[-{Stealth[length=3mm, width=2mm]},thick] (3,{sqrt(27)})--(1,{sqrt(27)});
%\draw[-{Stealth[length=3mm, width=2mm]},thick] (0,{sqrt(48)})--(-2,{sqrt(48)});\draw[-{Stealth[length=3mm, width=2mm]},thick]  (2,{sqrt(48)})--(0,{sqrt(48)});\draw[-{Stealth[length=3mm, width=2mm]},thick] (4,{sqrt(48)})--(2,{sqrt(48)});
\draw[-{Stealth[length=3mm, width=2mm]},thick]  (-6,0)--(-5,{sqrt(3)}); \draw[-{Stealth[length=3mm, width=2mm]},thick] (-4,0)--(-3,{sqrt(3)}); \draw[-{Stealth[length=3mm, width=2mm]},thick] (-2,0)--(-1,{sqrt(3)}); 
\draw[-{Stealth[length=3mm, width=2mm]},thick]  (-5,{sqrt(3)})--(-4,{sqrt(12)}); \draw[-{Stealth[length=3mm, width=2mm]},thick] (-3,{sqrt(3)})--(-2,{sqrt(12)}); \draw[-{Stealth[length=3mm, width=2mm]},thick] (-1,{sqrt(3)})--(0,{sqrt(12)}); 
\draw[-{Stealth[length=3mm, width=2mm]},thick]  (-4,{sqrt(12)})--(-3,{sqrt(27)}); \draw[-{Stealth[length=3mm, width=2mm]},thick] (-2,{sqrt(12)})--(-1,{sqrt(27)}); \draw[-{Stealth[length=3mm, width=2mm]},thick] (0,{sqrt(12)})--(1,{sqrt(27)}); 
\draw[thick]  (-2.5,+{sqrt(147)}/2)--(-3,{sqrt(27)}); \draw[thick] (-0.5,+{sqrt(149)}/2)--(-1,{sqrt(27)}); \draw[thick] (1.5,+{sqrt(149)}/2)--(1,{sqrt(27)});
\draw[-{Stealth[length=3mm, width=2mm]},thick]  (-5,{sqrt(3)})--(-4,0); \draw[-{Stealth[length=3mm, width=2mm]},thick] (-3,{sqrt(3)})--(-2,0); \draw[-{Stealth[length=3mm, width=2mm]},thick] (-1,{sqrt(3)})--(0,0); 
\draw[-{Stealth[length=3mm, width=2mm]},thick]  (-4,{sqrt(12)})--(-3,{sqrt(3)}); \draw[-{Stealth[length=3mm, width=2mm]},thick] (-2,{sqrt(12)})--(-1,{sqrt(3)}); \draw[-{Stealth[length=3mm, width=2mm]},thick] (0,{sqrt(12)})--(1,{sqrt(3)}); 
\draw[-{Stealth[length=3mm, width=2mm]},thick]  (-3,{sqrt(27)})--(-2,{sqrt(12)}); \draw[-{Stealth[length=3mm, width=2mm]},thick] (-1,{sqrt(27)})--(0,{sqrt(12)}); \draw[-{Stealth[length=3mm, width=2mm]},thick] (1,{sqrt(27)})--(2,{sqrt(12)}); 
\draw[-{Stealth[length=3mm, width=2mm]},thick]  (-1.5,+{sqrt(147)}/2)--(-1,{sqrt(27)}); \draw[-{Stealth[length=3mm, width=2mm]},thick] (0.5,+{sqrt(149)}/2)--(1,{sqrt(27)}); \draw[-{Stealth[length=3mm, width=2mm]},thick] (2.5,+{sqrt(149)}/2)--(3,{sqrt(27)}); 

\fill (-3,{sqrt(27)}) circle (3pt);\fill (-1,{sqrt(27)}) circle (3pt);\fill (1,{sqrt(27)}) circle (3pt);
 \fill (-4,0) circle (3pt);\fill (-6,0) circle (3pt);\ \fill (-4,0) circle (3pt);\fill (-2,0) circle (3pt); \fill (0,0) circle (3pt);
\fill (-5,{sqrt(3)}) circle (3pt);\fill (-3,{sqrt(3)}) circle (3pt);\fill (-1,{sqrt(3)}) circle (3pt);\fill (1,{sqrt(3)}) circle (3pt);
\fill (-4,+{sqrt(12)}) circle (3pt);\fill (-2,+{sqrt(12)}) circle (3pt); \fill (0,+{sqrt(12)}) circle (3pt); \fill (2,+{sqrt(12)}) circle (3pt); 
\fill (-3,{sqrt(27)}) circle (3pt);\fill (-1,{sqrt(27)}) circle (3pt);\fill (1,{sqrt(27)}) circle (3pt);\fill (3,{sqrt(27)}) circle (3pt);
\node at (-5,-0.3) {$f_{0,0,3}$};\node at (-3,-0.3) {$f_{1,0,3}$};\node at (-1,-0.3) {$f_{2,0,3}$};
\node at (-4,1.5) {$f_{0,1,3}$};\node at (-2,1.5) {$f_{1,1,3}$};\node at (0,1.5) {$f_{2,1,3}$};
\node at (-3,3.2) {$f_{0,2,3}$};\node at (-1,3.2) {$f_{1,2,3}$};\node at (1,3.2) {$f_{2,2,3}$};
\node at (-2,4.9) {$f_{0,3,3}$};\node at (0,4.9) {$f_{1,3,3}$};\node at (2,4.9) {$f_{2,3,3}$};

\node at (-5.5,0.7) {$f_{0,0,1}$}; \node at (-4.5,0.7) {$f_{0,0,2}$};\node at (-3.5,0.7) {$f_{1,0,1}$};\node at (-2.5,0.7) {$f_{1,0,2}$};\node at (-1.5,0.7) {$f_{2,0,1}$};\node at (-0.5,0.7) {$f_{2,0,2}$};
\node at (-4.5,2.5) {$f_{0,1,1}$}; \node at (-3.5,2.5) {$f_{0,1,2}$};\node at (-2.5,2.5) {$f_{1,1,1}$};\node at (-1.5,2.5) {$f_{1,1,2}$};\node at (-0.5,2.5) {$f_{2,1,1}$};\node at (0.5,2.5) {$f_{2,1,2}$};
\node at (-3.5,4.3) {$f_{0,2,1}$};\node at (-2.5,4.3) {$f_{0,2,2}$};\node at (-1.5,4.3) {$f_{1,2,1}$};\node at (-0.5,4.3) {$f_{1,2,2}$};\node at (0.5,4.3) {$f_{2,2,1}$};\node at (1.5,4.3) {$f_{2,2,2}$};
\draw[-,dashed] (240:1) --++ (60:8.5);
\end{tikzpicture}
\end{center}
\caption{$\operatorname{E}_2(Y_3)$}\label{fig:Y_3}
\end{figure}

%\begin{figure}[h]
%\begin{center}
%\begin{tikzpicture}[scale=0.7]
%\draw[gray,thick,shorten >=3pt,->] (-4,0)--(-6,0);\draw[thick,shorten >=3pt,->] (-6,0)--(-8,0);
%\draw[thick,shorten >=3pt,->] (-6,0)--(-5,{sqrt(3)});
%\draw[thick,shorten >=3pt,->] (-6,0)--(-5,-{sqrt(3)});
%\draw[gray, thick, shorten >=3pt,->] (-1,+{sqrt(3)})--(-0,0);\draw[thick,shorten >=3pt,->] (0,0)--(-2,0);
%\draw[thick,shorten >=3pt,->] (0,0)--(1,-{sqrt(3)});
%\draw[thick,shorten >=3pt,->] (0,0)--(1,+{sqrt(3)});
%\draw[gray, thick, shorten >=3pt,->] (5,-{sqrt(3)})--(6,0);\draw[thick,shorten >=3pt,->] (6,0)--(4,0);
%\draw[thick,shorten >=3pt,->] (6,0)--(7,{sqrt(3)});
%\draw[thick,shorten >=3pt,->] (6,0)--(7,-{sqrt(3)});
%\fill (-6,0) circle (3pt);\fill (0,0) circle (3pt);\fill (6,0) circle (3pt);\node at (-6.8,0.3) {\tiny{$q^2$}};\node at (-5.8,1) {\tiny{$0$}};\node at (-5.8,-0.8){\tiny{$0$}};\node at(-1,0.3) {\tiny{$q^2-q$}};\node at (0.6,-0.5) {\tiny{$q$}};\node at (0.7,0.5) {\tiny{$0$}}; \node at (6.2,0.8) {\tiny{$1$}}; \node at (7,-0.6) {\tiny{$q-1$}};\node at (5.1,0.3) {\tiny{$q^2-q$}};
%\end{tikzpicture}
%\end{center}
%\end{figure}
 The operator $T_{2,k}$ on the set of oriented edges of color 1 in $Y_k$ is defined by 
\begin{equation*}
\begin{split}
&T_{2,k}f_{m,n,1}:=\begin{cases}
(q^2-1)f_{0,n,2}+f_{0,n+1,1}& \text{ if } m=0\\
(q-1)f_{m,n,2}\\+(q^2-q)f_{m-1,n+1,3}+f_{m,n+1,1}&\text{ if } n\neq 0\\
\end{cases}\\
&T_{2,k}f_{m,n,2}:=\begin{cases}
(q^2-q)f_{m,0,3}+qf_{m+1,0,1}& \text{ if } n=0\\
(q^2-q)f_{m,n,3}+qf_{m+1,n-1,2}&\text{ if }n\neq 0,k-1\\
(q^2-q)f_{k-1,n,3} &\text{ if } n=k-1\\
\end{cases}\\
&T_{2,k}f_{m,n,3}:=\begin{cases}
q^2f_{0,0,1}& \text{ if } m=n=0\\
q^2f_{0,n-1,2}&\text{ if }m=0,n\neq 0\\
q^2f_{m-1,n,3}& \text{otherwise}.\\
\end{cases}\\
\end{split}
\end{equation*}

Under the definition of $T_{2,k}$ on $Y_k$, the matrix representation of $I-uT_{2,k}$ coincides with $I-uT_{1,k}$. Thus the determinant of $I-u^2T_{2,k}$ is the limit 
$$\lim_{N\rightarrow \infty} \det M_{k,N}(u^2)=\det A'_k(u^2),$$
where $M_{k,N}(u)$ is defined in \eqref{eq:1}. By \eqref{eq:2},
\begin{equation*}
\det(I-u^2T_{2,k})=\lim_{k\rightarrow \infty}\det (I-u^2T_{2,k})=\lim_{k\rightarrow\infty}\det A_k(u^2)=\frac{(1-q^3u^6)(1-q^6u^6)}{(1-q^4u^6)^2}.
\end{equation*}

\end{document}